\documentclass{amsart}

\usepackage[utf8]{inputenc}

\usepackage{amsmath,nccmath,amsthm,amssymb}
\usepackage{pstricks,pst-node,multido}
\usepackage[arrow,matrix]{xy}
\usepackage{a4wide}
\usepackage[linktocpage,unicode]{hyperref}

\title[Touchard-Riordan formulas, T-fractions, and Jacobi's triple product identity]{Touchard-Riordan formulas, T-fractions, \\ and Jacobi's triple product identity}
\author{Matthieu Josuat-Verg\`es}
\address{Fakultät für Mathematik, Universität Wien, AUSTRIA}
\email{matthieu.josuat-verges@univie.ac.at}
\author{Jang Soo Kim}
\address{School of Mathematics, University of Minnesota, USA}
\email{kimjs@math.umn.edu}
\thanks{Both authors were supported by the French National Research Agency ANR, grant ANR08-JCJC-0011.
   Moreover the first author was supported by the Austrian Science foundation FWF, START grant Y463.}
\subjclass[2000]{Primary: 05A19, 30B70, Secondary: 05A15, 33D45}
\keywords{Jacobi's triple product identity, continued fractions, Euler numbers, Genocchi numbers}
\date{}

\newtheorem{thm}{Theorem}[section]
\newtheorem{lem}[thm]{Lemma}
\newtheorem{prop}[thm]{Proposition}

\theoremstyle{definition}

\newtheorem{defn}{Definition}

\newtheorem{prob}{Problem}
\theoremstyle{remark}
\newtheorem{remark}{Remark}

\hypersetup{
    unicode=true,          
    pdftoolbar=true,        
    pdfmenubar=true,        
    pdffitwindow=false,     
    pdfstartview={FitH},    
    pdftitle={Touchard-Riordan formulas, T-fractions, and Jacobi's triple product identity},
    pdfauthor={Matthieu Josuat-Vergès and Jang-Soo Kim},
    pdfsubject={},
    pdfcreator={Creator},   
    pdfproducer={Producer}, 
    pdfkeywords={Jacobi's triple product identity, continued fractions, Euler numbers, Genocchi numbers},
    pdfnewwindow=true,      
    colorlinks=true,       
    linkcolor=red,          
    citecolor=blue,        
    filecolor=magenta,      
    urlcolor=cyan           
}

\newcommand{\cmo}{\genfrac{}{}{0pt}{}{}{-}}
\newcommand{\cpo}{\genfrac{}{}{0pt}{}{}{+}}

\newcommand{\y}{y^{-1}}
\newcommand\flr[1]{\left\lfloor #1\right\rfloor}
\newcommand\qint[1]{\left[ #1\right]_q}
\newcommand\yqint[1]{\left[ #1\right]_{y,q}}
\newcommand\yiqint[1]{\left[ #1\right]_{y^{-1},q}}
\newcommand\A{\mathcal{A}}
\newcommand\B{\mathcal{B}}
\newcommand\Gg{\mathcal{G}}

\newcommand\J{\mathcal{J}}
\renewcommand\U{\mathcal{U}}
\newcommand\V{\mathcal{V}}
\newcommand\D{\mathcal{D}}
\newcommand\wt{\operatorname{wt}}

\newcommand\one{\mathbf{1}}
\newcommand\zero{\mathbf{0}}
\newcommand\norm[1]{\lVert #1\rVert}

\newcommand\Fill{\mathbf{Fill}}
\newcommand\Shrink{\mathbf{Shrink}}
\newcommand\Stretch{\mathbf{Stretch}}
\newcommand\Remove{\mathbf{Remove}}
\newcommand\Ascend{\mathbf{Ascend}}
\newcommand\Descend{\mathbf{Descend}}

\def\dk{\delta_k}
\def\dkc.{$\delta_k$-configuration}
\def\dkpc.{$\delta_k^+$-configuration}
\def\dkmc.{$\delta_{k}^-$-configuration}
\def\HDp#1{\mathcal{H}\Delta_{#1}^+}
\def\Dp#1{\Delta_{#1}^+}
\def\Dm#1{\Delta_{#1}^-}

\def\ywt.{$(y,q)$-weight}
\newcommand\SCH{\mathcal{S}}
\newcommand\MSCH{\overline{\mathcal{S}}}
\newcommand\MD{\overline{\mathcal{D}}}
\newcommand\OP{\mathcal{OP}}

\def\sch.{Schr\"oder}

\psset{gridlabels=0pt,subgriddiv=1, gridwidth=.1pt}
\psset{linewidth=.1pt, unit=.6cm}
\psset{unit=12pt}

\newcount\ax \newcount\ay
\newcount\bx \newcount\by
\newcount\cx \newcount\cy
\newcount\dx \newcount\dy
\def\cell(#1,#2)[#3]{
\ax=#2 \ay=#1
\multiply\ay by-1
\bx=\ax \by=\ay
\cx=\ax \cy=\ay
\dx=\ax \dy=\ay
\advance\bx by-1
\advance\dy by1
\advance\cx by-1
\advance\cy by1
\psline (\dx,\dy)(\ax,\ay)(\bx,\by)
\rput(\number\cx.5,
\ifnum\cy=0 -0.5\else\number\cy.5\fi){#3}
}
\def\gcell(#1,#2)[#3]{
\ax=#2 \ay=#1
\multiply\ay by-1
\bx=\ax \by=\ay
\cx=\ax \cy=\ay
\dx=\ax \dy=\ay
\advance\bx by-1
\advance\dy by1
\advance\cx by-1
\advance\cy by1
\psframe[linestyle=none,fillstyle=solid,fillcolor=gray!40!white](\ax,\ay)(\cx,\cy)
\rput(\number\cx.5,
\ifnum\cy=0 -0.5\else\number\cy.5\fi){#3}
}
\def\Gcell(#1,#2)[#3]{
\gcell(#1,#2)[#3] \cell(#1,#2)[#3]
}
\def\psrow(#1,#2){\multido{\i=1+1}{#2}{\cell(#1,\i)[]}}
\def\psdk#1{
\psline(#1,0)(0,0)(0,-#1)
\multido{\n=1+1}{#1}{
  \ax=#1 \advance\ax by 1
  \advance\ax by -\n
  \psrow(\n,\ax)}
}
\def\rowseg#1[#2]{\rput(0,-#1){\rput(0,.5){
      \psline[linewidth=1.5pt](0,0)(#2,0)}}}
\def\rowarr#1[#2]{\rput(0,-#1){\rput(0,.5){
      \psline[linewidth=1.5pt,arrowsize=.5, arrowlength=.6]{->}(0,0)(#2,0)}}}
\def\colseg#1[#2]{\rput(#1,0){\rput(-.5,0){
      \psline[linewidth=1.5pt](0,0)(0,-#2)}}}
\def\colarr#1[#2]{\rput(#1,0){\rput(-.5,0){
      \psline[linewidth=1.5pt,arrowsize=.5, arrowlength=.6]{->}(0,0)(0,-#2)}}}
\def\harrow(#1,#2)[#3]{\rput(#2,-#1){\rput(-1,.5){
      \psline[linewidth=1.5pt,arrowsize=.5, arrowlength=.6]{->}(0,0)(#3,0)}}}
\def\varrow(#1,#2)[#3]{\rput(#2,-#1){\rput(-.5,1){
      \psline[linewidth=1.5pt,arrowsize=.5, arrowlength=.6]{->}(0,0)(0,-#3)}}}
\def\hzero(#1,#2){\rput(#2,-#1){\rput(-1,.5){\psarc*(0,0){.2}{-90}{90}}}}
\def\vzero(#1,#2){\rput(#2,-#1){\rput(-.5,1){\psarc*(0,0){.2}{180}{0}}}}
\def\hzerow(#1,#2){\rput(#2,-#1){\rput(-1,.5){\psarc(0,0){.2}{-90}{90}}}}
\def\vzerow(#1,#2){\rput(#2,-#1){\rput(-.5,1){\psarc(0,0){.2}{180}{0}}}}

\def\UP(#1,#2)[#3]{\rput(#1,#2){\psline(0,0)(1,1) 
}}
\def\DW(#1,#2)[#3]{\rput(#1,#2){\psline(0,0)(1,-1) 
}}
\def\MUP(#1,#2)[#3]{\rput(#1,#2){\psline[linewidth=2pt](0,0)(1,1)
}}
\def\MDW(#1,#2)[#3]{\rput(#1,#2){\psline[linewidth=2pt](0,0)(1,-1)
}}

\begin{document}

\begin{abstract}
Touchard-Riordan--like formulas are some expressions appearing in enumeration problems and as moments
of orthogonal polynomials. We begin this article with a new combinatorial approach to prove these
kind of formulas, related with integer partitions. This gives a new perspective on the original result
of Touchard and Riordan. But the main goal is to give a combinatorial proof of a 
Touchard-Riordan--like formula for $q$-secant numbers discovered by the first author.
An interesting limit case of these objects can be directly interpreted in terms of partitions, so that we
obtain a connection between the formula for $q$-secant numbers, and a particular case of Jacobi's
triple product identity.

Building on this particular case, we obtain a ``finite version'' of the triple
product identity.  It is in the form of a finite sum which is given a
combinatorial meaning, so that the triple product identity can be obtained by
taking the limit. Here the proof is non-combinatorial and relies on a functional
equation satisfied by a T-fraction. Then from this result on the triple product identity,
we derive a whole new family of Touchard-Riordan--like formulas whose
combinatorics is not yet understood.  Eventually, we prove a
Touchard-Riordan--like formula for a $q$-analog of Genocchi numbers, which is
related with Jacobi's identity for $(q;q)^3$ rather than the triple product
identity.
\end{abstract}

\maketitle

\setcounter{tocdepth}{1}
\tableofcontents

\part*{Introduction}

The original result of Touchard \cite{touchard}, later given more explicitly
by Riordan \cite{riordan}, answers the combinatorial problem of counting chord diagrams according to the number
of crossings. It has also been stated in terms of continued fractions by Read \cite{read}, so that the Touchard-Riordan
formula is:
\begin{equation} \label{tourio}
  [z^n] \left(  \frac 11 \cmo \frac{[1]_qz}{1} \cmo \frac{[2]_qz}{1} \cmo \dotsb \right)   =
  \frac{1}{(1-q)^n} \sum_{k=0}^n \left(\tbinom{2n}{n-k}-\tbinom{2n}{n-k-1}\right)(-1)^k q^{\binom{k+1}2},
\end{equation}
where $[z^n]$ is the operator that extracts the coefficient of $z^n$ in a
series, $[n]_q$ denotes $(1-q^n)/(1-q)$, and we use the notation for continued
fractions as in \eqref{notationfrac}. A combinatorial proof has been given by Penaud \cite{penaud}.
Recently, several variants have been
derived.  In particular, using continued fractions and basic hypergeometric
series, the first author \cite{josuat1} proved the following formula in a slightly different form:
\begin{equation} \label{qsec}
 [z^n] \left(  \frac 11 \cmo \frac{[1]_q^2z}{1} \cmo \frac{[2]_q^2z}{1} \cmo \dotsb \right)  =
  \frac{1}{(1-q)^{2n}} \sum_{k=0}^n \left(\tbinom{2n}{n-k}-\tbinom{2n}{n-k-1}  \right) q^{k(k+1)}\sum_{i=-k}^{k}(-q)^{-i^2}.
\end{equation}
This quantity will be referred as $q$-secant number, for it is a $q$-analog of the integers having
exponential generating function $\sec(z)$, and it is related with alternating permutations.
Our interest in these kind of formulas relies on their combinatorial meaning, and
also on that they are moments of orthogonal polynomials \cite{ismail,josuat2}.

\medskip

Some of the proofs of \eqref{tourio} and \eqref{qsec} are related with T-fractions, which are a particular kind
of continued fractions as in Definition~\ref{def_frac} below. Indeed, the common factor
$\big(\tbinom{2n}{n-k}-\tbinom{2n}{n-k-1} \big)$ can be fully explained by the link between S-fractions
and T-fractions (see Lemmas~\ref{lem2} and \ref{lem2bis}). The T-fractions
appear occasionally in combinatorics \cite{roblet} but much less than S-fractions or J-fractions.
They are central in this work: the first part of this article relies on the Roblet-Viennot interpretation of T-fractions 
in terms of Dyck paths~\cite{roblet}, and the second part, less combinatorial, relies on some functional equations
satisfied by T-fractions.
For example, the T-fractions were used in \cite{josuat1} to prove \eqref{qsec}. In the case of \eqref{tourio}, the link with T-fractions
was known by Cigler, Flajolet and Prodinger \cite{flajolet}.

\medskip

In the first part of this article, we have a new combinatorial approach based on the Roblet-Viennot 
interpretation of T-fractions \cite{roblet}. The first idea is to use the simple bijection between Dyck paths and Ferrers
diagrams that fit in a staircase diagram. The weighted Dyck paths considered here lead to define some 
objects that we call \dkc.s.
Using these objects we will see that the proof of \eqref{tourio} reduces to a classical construction on integer partitions, known
as Vahlen's involution \cite{vahlen}. But the main goal of this first part is to give a combinatorial proof of \eqref{qsec},
and to do this we show that \dkc.s give a combinatorial model whose weight sum is $\sum_{i=-k}^{k}(-q)^{i^2}$.
Besides, a nice feature of these \dkc.s is that in a limiting case they have a simple meaning in terms of integer
partitions, so that letting $k$ tend to infinity gives:
\begin{equation}
  \label{eq:gauss}
 \prod_{i\geq1} \frac{1-q^{i}}{1+q^{i}}=\sum_{i=-\infty}^{\infty} (-q)^{i^2},
\end{equation}
which is the special case $y=-1$ of Jacobi's triple product identity:
\begin{equation}
  \label{eq:jtp}
\prod_{n\geq1} (1-q^{2n}) (1+y q^{2n-1}) (1+y^{-1} q^{2n-1}) =
\sum_{n=-\infty}^{\infty} y^{n} q^{n^2}.
\end{equation}

Jacobi's triple product identity is ubiquitous in various areas of
mathematics and especially in analytical number theory,  quite a lot of
different proofs, generalizations and variants are known, see for example
\cite{alladi,berndt,schlosser,warnaar} and lots of references therein.
See \cite[Chapter 1]{berndt} for a general reference about this kind of identities.

So it is highly desirable to find a result similar to \eqref{qsec} whose
limiting case gives the triple product identity.
This turns out to be possible in the following form:

\begin{thm}\label{thm:main}
We define $\yqint{n}=(1+yq^n)/(1-q)$. Then
\begin{multline} \label{fjtp}
 [z^n] \left(  \frac 11 \cmo \frac{\yqint{1}\yiqint{1}z}{1} \cmo 
\frac{\qint{2}^2z}{1} \cmo
\frac{\yqint{3}\yiqint{3}z}{1} \cmo
\frac{\qint{4}^2z}{1} \cmo \dots \right)  \\
= \frac{1}{(1-q)^{2n}} \sum_{k=0}^n \left(\tbinom{2n}{n-k}-\tbinom{2n}{n-k-1}  \right) q^{k(k+1)}\sum_{i=-k}^{k} y^i q^{-i^2}.
\end{multline}
\end{thm}

Note that if $y=-1$, we get \eqref{qsec}. An equivalent formulation without the factor 
$\big(\tbinom{2n}{n-k}-\tbinom{2n}{n-k-1}  \big)$ can be given by considering a T-fraction,
see Theorem~\ref{jtp_th}. Another equivalent formulation is that the sum
\begin{equation} \label{sumj}
   \sum_{i=-k}^{k} y^i q^{k(k+1)-i^2}
\end{equation}
can be given a combinatorial meaning in terms of paths (either Roblet-Viennot--style Dyck paths, or Schröder paths),
in such a way that the limit as $k\to\infty$ gives the triple product identity. For this reason, we call this result a
``finite version'' of the triple product identity.

\bigskip

In the second part of this article, we prove Theorem~\ref{thm:main} (more precisely, we prove the equivalent form
as in Theorem~\ref{jtp_th}) by giving a functional equation, satisfied on one side by the generating function of
\eqref{sumj}, and on the other side by a T-fraction. This proof is not at all combinatorial.

By taking some specializations in Theorem~\ref{thm:main}, we can obtain a whole new family of
Touchard-Riordan--like formulas. Generalizing the case of $q$-secant numbers, these new formulas
concern in particular $q$-analogs of integers having exponential generating function
\[
 \frac{  \cos(az)  }{ \cos(bz)  },
\]
where $a$ and $b$ satisfy certain conditions, see Theorem~\ref{theoS} for details.
These $q$-analogs might have a combinatorial meaning, however it is still to be investigated
and some open problems appear.

Eventually, we also prove a formula for a $q$-analog of Genocchi numbers (which have exponential
generating function $z\tan\frac z2$). Instead of the triple product identity, this one is related with the following, also
due to Jacobi:
\[
    \prod_{j=1}^\infty (1-q^j)^3 = \sum_{i=0}^{\infty} (-1)^i (2i+1) q^{\binom{i+1}2 }.
\]

The organization of this article should be clear from the table of contents at the beginning,
and the precisions given in the introduction.

\bigskip

\part{\texorpdfstring{Combinatorial proofs of Touchard-Riordan--like formulas and a particular case of triple product identity}{Combinatorial proofs of Touchard-Riordan–like formulas and a particular case of triple product identity}}

\section{Preliminaries}
\label{sec:prelem}

We define here S-fractions and T-fractions as certain formal power series in $z$, we give their combinatorial
interpretations in terms of paths, and give a lemma which is a basic tool to obtain Touchard-Riordan--like formulas.

\begin{defn} \label{def_frac}
We will use the space-saving notation for continued fractions:
\begin{equation} \label{notationfrac}
  \frac{a_0}{b_0} \cmo \frac{a_1}{b_1} \cmo \frac{a_2}{b_2} \cmo \dotsb =
  \cfrac{a_0}{b_0-
    \cfrac{a_1}{b_1-
      \cfrac{a_2}{ b_2- \genfrac{}{}{0pt}{0}{\phantom a}{\ddots}
      }
    }
  }.
\end{equation}
And to any sequence $\lambda=\{\lambda_n\}_{n\geq1}$, we associate the S-fraction $S_\lambda(z)$
and  the T-fraction $T_\lambda(z)$:
\[
  S_\lambda(z) =  \frac{1}{1} \cmo \frac{\lambda_1 z}{1} \cmo \frac{\lambda_2 z}{1}
                  \cmo \frac{\lambda_3 z}{1} \cmo \dotsb, \qquad 
 T_\lambda(z) =  \frac{1}{1+z} \cmo \frac{\lambda_1 z}{1+z} \cmo \frac{\lambda_2 z}{1+z}
                  \cmo \frac{\lambda_3 z}{1+z} \cmo \dotsb. 
\]
\end{defn}

The combinatorial interpretation of S-fractions in terms of weighted Dyck paths is widely known
(see for example \cite{goulden}), but the analogous result for T-fractions is not as common.
Roblet and Viennot \cite{roblet} showed that T-fractions are related with a particular
kind of weighted Dyck path with some conditions on each peak. Equivalently,
we can see T-fractions as related with Schröder paths.

\begin{defn}
  A \emph{\sch. path} of length $2n$ is a path from $(0,0)$ to $(2n,0)$ in
  $\mathbb{N}^2$ with three kinds of steps: an up step $(1,1)$, a down step
  $(1,-1)$, and a horizontal step $(2,0)$. A \emph{Dyck path} is a Schröder path
  with no horizontal step. Let $\SCH_n$ be the set of Schröder paths of length
  $2n$, and $\D_n\subset \SCH_n$ be the subset of Dyck paths of length $2n$.
\end{defn}

The continued fractions we consider here will be related with paths having
weights of the form $1-q^h$, $1+yq^h$, or $1+y^{-1}q^h$ for some $h$. It is
natural to distinguish two kinds of steps, an ``unweighted kind'' for the term $1$
and a ``weighted kind'' for the other terms $-q^h$, {\it etc}. This leads to the following
definition.

\begin{defn}
  A \emph{marked \sch. path} is a \sch. path in
  which each up step and down step may be marked. Let $\MSCH_n$ (resp.~$\MD_n$)
  denote the set of marked \sch. paths (resp. marked Dyck paths) of length $2n$.
  Let $\MD^*_n$ denote the subset of
  $\MD_n$ consisting of the marked Dyck paths without any ``marked peak'', {\it i.e.} an up step
  immediately followed by a down step, both marked.
\end{defn}

Note that we have $\D_n\subset\MD^*_n$  by identifying a Dyck path with a marked Dyck path having no marked
step. 

\begin{defn}
Given sequences $\A=(a_1,a_2,\dots)$, $\B=(b_1,b_2,\dots)$ and a marked
Schröder path $p$, we define the \emph{weight} $\wt(p;\A,\B)$ to be the product of
$a_h$ (resp.~$b_h$) for each unmarked up step (resp.~unmarked down step) between
height $h$ and $h-1$, and $-1$ for each horizontal step
 (hence each marked step has weight 1).
We will use the following sequences:
\begin{align*}
\one &= (1,1,1,\dots),    &  \zero  &= (0,0,0,\dots), \\
\U &=(\qint{1}, \qint{2}, \ldots), & \V &=(1-q, 1-q^2, \ldots).
\end{align*}
\end{defn}

The following lemma is the combinatorial interpretation of continued fractions in terms of paths.

\begin{lem}\label{lem1}
 Let $\A=(a_1,a_2,\dots)$, $\B=(b_1,b_2,\dots)$ be two sequences and
  $\lambda_h=a_hb_h$. Then
\begin{align}
 S_\lambda(z) &= \sum_{n=0}^\infty z^n \sum_{ p \in \D_n} \wt(p;\A,\B),   \label{expS}    \\
 T_\lambda(z) &= \sum_{n=0}^\infty z^n \sum_{ p \in \SCH_n } \wt(p;\A,\B)
                       = \sum_{n=0}^\infty z^n \sum_{ p \in \MD^*_n } \wt(p;\A-\one,\B-\one), \label{expT}
\end{align}
where $\A-\one$ means the sequence $(a_1-1,a_2-1,\dots)$.
\end{lem}

\begin{proof}
The first identity \eqref{expS} is a classical result, see for example \cite{goulden}. It follows essentially
from the fact that if $f$ is a generating function counting some objects, $1/(1-f)$ counts sequences of 
objects. The same method proves the equality of the first two expressions in \eqref{expT}.

The equality between the first and third expressions in \eqref{expT} is a result of Roblet and Viennot~\cite{roblet}.
Besides, a direct proof of the equality between the second and third expressions in \eqref{expT} can be done.
First note that
\[
\sum_{p\in \SCH_n} \wt(p;\A,\B)=\sum_{p\in \MSCH_n} \wt(p;\A-\one,\B-\one).
\]
There is a sign-reversing involution on $\MSCH_n$ such that the set of fixed points is $\MD^*_n$.
This involution just exchanges a horizontal step (this has weight $-1$) with a marked peak (this has weight $1$). Hence:
\begin{equation}  \label{eq:2}
\sum_{p\in \SCH_n} \wt(p;\A,\B)= \sum_{p\in \MD^*_n} \wt(p;\A-\one,\B-\one).
\end{equation}
Thus we get  \eqref{expT}.
\end{proof}

The following lemma gives a relation between the coefficients of an
S-fraction and a T-fraction. This is the first step in our proofs of the
Touchard-Riordan--like formulas.

\begin{lem} \label{lem2}
For any sequences $\A$ and $\B$ we have
\begin{align}
 \sum_{p\in \D_n} \wt(p;\A,\B)
 &= \sum_{k=0}^n \left( \tbinom{2n}{n-k} - \tbinom{2n}{n-k-1} \right) \sum_{p\in \SCH_k} \wt(p;\A,\B)     \label{lem2S}     \\
 &= \sum_{k=0}^n \left( \tbinom{2n}{n-k} - \tbinom{2n}{n-k-1} \right) \sum_{p\in \MD^*_k} \wt(p;\A-\one,\B-\one).   \label{lem2D}
\end{align}
\end{lem}

\begin{proof}
  Firstly, the equivalence of \eqref{lem2S} and \eqref{lem2D} follows from \eqref{eq:2}.
  Secondly, an inclusion-exclusion between Schröder paths and Dyck paths gives a quick proof of \eqref{lem2S}.
  Indeed, a Schröder path of length $2n$ can be obtained from a Dyck path of length $2k$ by inserting $n-k$
  horizontal steps, and there are $\binom{n+k}{n-k}$ ways to do so. This shows
\[
  \sum_{p\in \SCH_n} \wt(p;\A,\B)
  = \sum_{k=0}^n  (-1)^{n-k} \binom{n+k}{n-k}   \sum_{p\in \D_k} \wt(p;\A,\B).
\]
It remains only to inverse the (lower-triangular) matrix with coefficients $(-1)^{n-k} \binom{n+k}{n-k}$ to obtain \eqref{lem2S}.
See Riordan's book \cite{riordan68} for this kind of inverse relations: the present one appears in Chapter~2, Section~2.4,
Equation~(12).
\end{proof}

The previous lemma can be stated in terms of continued fractions via Lemma~\ref{lem1}.
We will do this explicitly in Lemma~\ref{lem2bis} where we give a proof based on continued fractions.

To complete these preliminaries, let us briefly describe a method to prove bijectively the previous lemma,
the idea is to generalize Penaud's decomposition from \cite{penaud} and this can be describes as
follows. Here, a \emph{Dyck prefix} is a path in $\mathbb{N}^2$ from the origin
to any point consisting of up steps and down steps.
Each $p\in\MD_n$ can be uniquely decomposed into $(\ell,p')$ where $\ell$ is a Dyck
prefix of length $2n$ ending at height $2k$ and $p'\in\MD^*_k$ for some
$k$, and in such a way that for any sequences $\A$ and $\B$, we have $\wt(p;\A,\B) =
\wt(p';\A,\B)$. It is easy to show that the number of Dyck prefixes of length $2n$ ending at
height $2k$ is equal to $\binom{2n}{n-k} - \binom{2n}{n-k-1}$.
Thus, from this decomposition,  we obtain bijectively
\begin{equation}
  \label{eq:1}
\sum_{p\in \MD_n} \wt(p;\A,\B) =
\sum_{k=0}^n \left( \tbinom{2n}{n-k} - \tbinom{2n}{n-k-1} \right)
\sum_{p\in \MD^*_k} \wt(p;\A,\B).
\end{equation}
Penaud actually defined his decomposition in terms of trees. The generalization in terms of paths is
described in \cite{rubey}. A map $\phi$ sending $p \in \MD_k(\A,\B)$ to a pair
$( \ell , p')$ is as follows:  We obtain $p'$ by deleting all sub-Dyck paths
of $p$ consisting of marked steps (thus $p'$ has no marked peak). 
Then $\ell$ is obtained from $p$ by replacing each down step of $p'$ with an up step.
See \cite{rubey} for more details. 

\section{Use of integer partitions: case of the Touchard-Riordan formula}
\label{sec:anoth-comb-proof}

Using Lemmas~\ref{lem1} and \ref{lem2}, the left-hand side of \eqref{tourio} is
\[
  \sum_{p\in\D_n} \wt(p;\U,\one) = \tfrac{1}{(1-q)^{n}}
  \sum_{p\in\D_n} \wt(p;\V,\one) = \tfrac{1}{(1-q)^{n}}
\sum_{k=0}^n \left( \tbinom{2n}{n-k} - \tbinom{2n}{n-k-1} \right)
\sum_{p\in \MD^*_k} \wt(p;\V-\one,\zero).
\]
To prove \eqref{tourio}, it remains only to show that 
\begin{equation}
  \label{eq:13}
\sum_{p\in \MD^*_k} \wt(p;\V-\one,\zero)= (-1)^kq^{\binom{k+1}2}.
\end{equation}
Penaud \cite{penaud} gave a sign-reversing involution by going through a series of objects
such as trees and parallelogram polyominoes. We give here another way to do so
in terms of partitions (but it can be checked that the involution we present is
equivalent to Penaud's).

Note that it is sufficient to consider $p\in\MD^*_k$ without unmarked down steps
because $\wt(p;\V-\one,\zero)=0$ if there is an unmarked down step.  Now we
introduce a combinatorial object corresponding to such $p\in\MD^*_k$.

A \emph{partition} is a weakly decreasing sequence
$\lambda=(\lambda_1,\lambda_2,\dots,\lambda_r)$ of positive integers called
\emph{parts}. We will sometimes identify a partition $\lambda$ with its
\emph{Ferrers diagram}. For two partitions $\lambda$ and $\mu$, we
write $\lambda\subset \mu$ if the Ferrers diagram of $\mu$ contains that of
$\lambda$. In this case we define $\mu/\lambda$ to be the set theoretic
difference of the Ferrers diagrams, see Figure~\ref{fig:Ferrers}. 

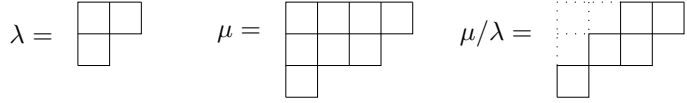
\begin{figure}
  \centering
 \begin{pspicture}(0,0)(4,-3)
   \rput[r](3.5,-1){$\lambda=$}
  \end{pspicture}
  \begin{pspicture}(0,0)(2,-3)
    \cell(1,1)[] \cell(1,2)[] 
    \cell(2,1)[] 
    \psline(0,-2)(0,0)(2,0)
  \end{pspicture}
 \begin{pspicture}(0,0)(4,-3)
   \rput[r](3.5,-1){$\mu=$}
  \end{pspicture}
  \begin{pspicture}(0,0)(4,-3)
    \cell(1,1)[] \cell(1,2)[] \cell(1,3)[] \cell(1,4)[] 
    \cell(2,1)[] \cell(2,2)[] \cell(2,3)[] \cell(3,1)[] 
    \psline(0,-3)(0,0)(4,0)
  \end{pspicture}
 \begin{pspicture}(0,0)(4,-3)
   \rput[r](3.5,-1){$\mu/\lambda=$}
  \end{pspicture}
  \begin{pspicture}(0,0)(4,-3)
    \cell(1,3)[] \cell(1,4)[] 
    \cell(2,2)[] \cell(2,3)[] \cell(3,1)[] 
    \psline(0,-3)(0,-2)(1,-2)(1,-1)(2,-1)(2,0)(4,0)
    \psline[linestyle=dotted,linewidth=.5pt](0,-2)(0,0)(2,0)
    \psline[linestyle=dotted,linewidth=.51pt](0,-1)(1,-1)(1,0)
 \end{pspicture}
  \caption{The Ferrers diagrams of $\lambda=(2,1)$, $\mu=(4,3,1)$, and $\mu/\lambda$.}
  \label{fig:Ferrers}
\end{figure}

We denote by $\delta_k$ the staircase partition $(k,k-1,\ldots,1)$.

\begin{defn}
  A \emph{half \dkpc.}  is a pair $(\lambda,A)$ of a partition
  $\lambda\subset\delta_{k-1}$ and a set $A$ of arrows each of which occupies a
  whole row of $\dk/\lambda$ such that no outer corner of $\dk/\lambda$ is
  occupied by an arrow. Here, by an outer corner we mean a cell $c\in
  \dk/\lambda$ such that $\lambda\cup c$ is a partition.  The \emph{length} of
  an arrow is the number of cells it occupies. We denote by $\HDp{k}$ the set of
  half \dkpc.s. The \emph{$q$-weight} of a half \dkpc. $C=(\lambda,A)$ is
  defined by
\[\wt_q(C) = (-1)^{|A|} q^{|\lambda| + \norm{A}},\]
where $\norm{A}$ is the sum of the arrow lengths.
\end{defn}

For example, the half \dkpc. in Figure~\ref{fig:dyckpath} (left) has $q$-weight
$\wt_q(C) = (-1)^3 q^{8+3+3+2}$. 

\begin{figure}
  \centering
\begin{pspicture}(0,0)(7,-7)
\gcell(1,1)[]\gcell(1,2)[]\gcell(1,3)[]\gcell(1,4)[]
\gcell(2,1)[]\gcell(2,2)[]
\gcell(3,1)[]\gcell(3,2)[]
\psdk7
\harrow(5,1)[3] \harrow(6,1)[2] \harrow(3,3)[3] 
\end{pspicture} 
\begin{pspicture}(-4,-1.5)(14,5)
\psgrid[gridcolor=gray](0,0)(14,4)
\rput(-2,2){$\Longleftrightarrow$}
\UP(0,0)[] \MUP(1,1)[$-q^2$] \MUP(2,2)[] \UP(3,3)[]
\MDW(4,4)[$-q^4$] \MDW(5,3)[] \MUP(6,2)[] \UP(7,3)[$-q^4$]
\MDW(8,4)[$-q^4$] \MDW(9,3)[] \UP(10,2)[$-q^3$] \MDW(11,3)[$-q^3$]
\MDW(12,2)[] \MDW(13,1)[$-q$]
\end{pspicture}
\caption{A half \dkpc. $(\lambda,A)$ and the corresponding marked Dyck path,
  where the Ferrers diagram of $\lambda$ is shaded (left) and the marked steps
  are the thicker steps (right).}
  \label{fig:dyckpath}
\end{figure}
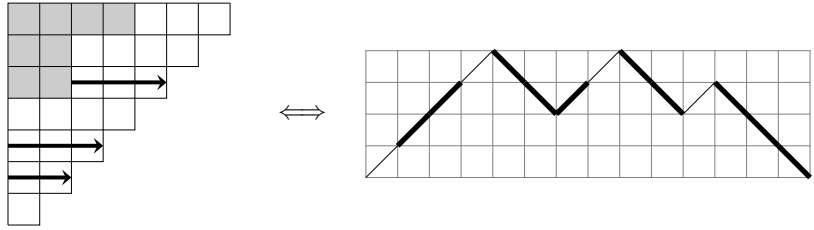

There is a simple correspondence between $(\lambda,A)\in \HDp k$ and
$p\in\MD^*_k$ without unmarked down steps as follows.  The border between
$\lambda$ and $\delta_k/\lambda$ is the Dyck path $p$ (rotated $45^\circ$) and
there is an arrow occupying a row of $\delta_k/\lambda$ if and only if the
corresponding up step in $p$ is marked. See Figure~\ref{fig:dyckpath}.  It is
not difficult to show that in this correspondence we have
\[
\wt(p,\V-\one,\zero) = (-1)^{k} q^{\binom{k+1}{2}} \wt_{q^{-1}}(\lambda,A). 
\]
Thus the following proposition implies \eqref{eq:13}.

\begin{prop}\label{prop:half_dk}
For $k\geq0$, we have
\[
  \sum_{(\lambda,A)\in\HDp k} \wt_q(\lambda,A) = 1.
\]
\end{prop}

In order to prove the above proposition we will use overpartitions introduced by
Corteel and Lovejoy \cite{corteel}.  An \emph{overpartition} is a partition in
which the last occurrence of an integer may be overlined. Let $\OP_k$ denote the
set of overpartitions whose underlying partitions are contained in
$\delta_{k-1}$. We will represent $\mu\in \OP_k$ as the Ferrers diagram of $\mu$
inside $\delta_k$, where each row corresponding to an overlined part is replaced
with an arrow of the same length. For example, the right diagram in
Figure~\ref{fig:overpartition} represents the overpartition $(\overline{6},
\overline 4, 3, \overline 3, 1, \overline1)$.

Now we will construct a bijection $\psi_1:\HDp k \to \OP_k$.  For
$(\lambda,A)\in\HDp k$, we define $\psi_1(\lambda,A)$ as follows. Firstly, for
each arrow $u$, we delete the cells of $\lambda$ in the row of $\delta_k$
containing $u$ and make $u$ occupy the whole row of $\delta_k$. See the left and
the middle diagrams in Figure~\ref{fig:overpartition}.  Note that because of the
condition on $(\lambda,A)\in\HDp k$ that no outer conner is occupied by an
arrow, this process is invertible. Secondly, we sort the remaining rows of
$\lambda$ and the arrows to get an overpartition. More precisely, whenever there
is a row of $\lambda$ immediately followed by an arrow of longer length, we
exchange the row of $\lambda$ and the arrow as shown below:
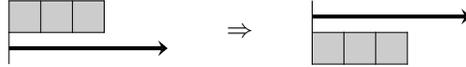
\begin{center}
  \begin{pspicture}(0,0)(5,-2)
    \gcell(1,1)[] \gcell(1,2)[] \gcell(1,3)[]
    \harrow(2,1)[5] \psgrid(0,0)(3,-1) \psline(0,-1)(0,-2)
  \end{pspicture}
 \begin{pspicture}(0,0)(4,-2)
   \rput(2,-1){$\Rightarrow$}
  \end{pspicture}
 \begin{pspicture}(0,0)(5,-2)
    \gcell(2,1)[] \gcell(2,2)[] \gcell(2,3)[]
    \harrow(1,1)[5] \psgrid(0,-1)(3,-2) \psline(0,0)(0,-1)
  \end{pspicture}
\end{center}
For example, if we apply this process to the middle diagram in
Figure~\ref{fig:overpartition}, then we get the right diagram there. By the
conditions on $(\lambda,A)\in \HDp k$, it is easy to see that the resulting
diagram is an overpartition contained in $\delta_{k-1}$. We define
$\psi_1(\lambda,A)$ to be the resulting overpartition.

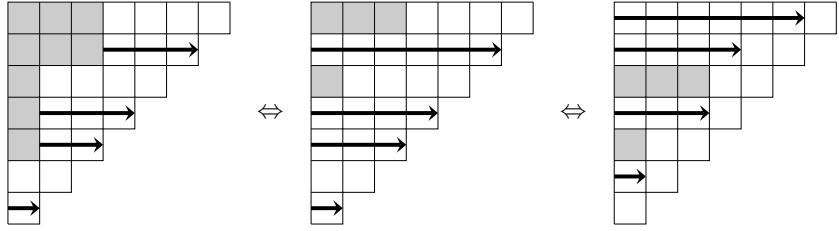
\begin{figure}
  \centering
\begin{pspicture}(0,0)(7,-7)
\gcell(1,1)[]\gcell(1,2)[]\gcell(1,3)[]
\gcell(2,1)[]\gcell(2,2)[]\gcell(2,3)[]
\gcell(3,1)[]
\gcell(4,1)[]
\gcell(5,1)[]
\psdk7
\harrow(2,4)[3] \harrow(4,2)[3] \harrow(5,2)[2]  \harrow(7,1)[1]
\end{pspicture} 
\begin{pspicture}(0,0)(2,-7)
\rput(1,-3.5){$\Leftrightarrow$}
\end{pspicture} 
\begin{pspicture}(0,0)(7,-7)
\gcell(1,1)[]\gcell(1,2)[]\gcell(1,3)[]
\gcell(3,1)[]
\psdk7
\harrow(2,1)[6] \harrow(4,1)[4] \harrow(5,1)[3]  \harrow(7,1)[1]
\end{pspicture} 
\begin{pspicture}(0,0)(2,-7)
\rput(1,-3.5){$\Leftrightarrow$}
\end{pspicture} 
\begin{pspicture}(0,0)(7,-7)
\gcell(3,1)[]\gcell(3,2)[]\gcell(3,3)[]
\gcell(5,1)[]
\psdk7
\harrow(1,1)[6] \harrow(2,1)[4] \harrow(4,1)[3]  \harrow(6,1)[1]
\end{pspicture} 
 \caption{A half \dkpc. (left) and the corresponding overpartition (right). One
   can obtain the left diagram from the middle one using the fact that no outer
   corner is occupied by an arrow.}
  \label{fig:overpartition}
\end{figure}

\begin{lem}\label{lem:ov}
  The map $\psi_1:\HDp k\to \OP_k$ is a bijection. Moreover, if $\psi_1(\lambda,A)=\mu$, then
\[ \wt_q(\lambda,A) = (-1)^{s(\mu)} q^{|\mu|}\]
where $s(\mu)$ is the number of overlined parts in $\mu$. 
\end{lem}
\begin{proof}
  We construct the inverse map of $\psi_1$ as follows. Let $\mu\in \HDp k$. We
  can assume that $\mu$ has $k$ parts by adding parts equal to $0$ if necessary.
  In the Ferrers diagram of $\mu$ with overlined parts replaced with arrows we
  exchange two consequences rows if the following conditions hold: (1) the upper
  row is an arrow and the lower row is not an arrow, and (2) after exchanging
  these rows the resulting diagram is still contained in
  $\delta_k$. Pictorially, this process is as following:
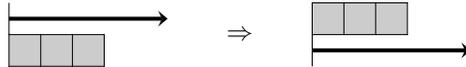
\begin{center}
 \begin{pspicture}(0,0)(5,-2)
    \gcell(2,1)[] \gcell(2,2)[] \gcell(2,3)[]
    \harrow(1,1)[5] \psgrid(0,-1)(3,-2) \psline(0,0)(0,-1)
  \end{pspicture}
\begin{pspicture}(0,0)(4,-2)
   \rput(2,-1){$\Rightarrow$}
  \end{pspicture}
  \begin{pspicture}(0,0)(5,-2)
    \gcell(1,1)[] \gcell(1,2)[] \gcell(1,3)[]
    \harrow(2,1)[5] \psgrid(0,0)(3,-1) \psline(0,-1)(0,-2)
  \end{pspicture}
\end{center}
We do this process until there are no such consecutive rows.

Since $\mu\subset \delta_{k-1}$ and each overlined part is always followed by a
part of smaller size, each arrow must be moved downwards at least once. If an
arrow is moved downwards, then it is still followed by a part of smaller
size. Thus each arrow will be moved downwards until it occupies a whole row of
$\delta_k$. Moreover, at the end, each arrow immediately follows a part of
smaller size or another arrow. We then add the first several cells in the row
occupied by an arrow to the Ferrers diagram of $\mu$ (and decrease the length of
the arrow by the same number) so that this row and the previous row have the
same number of cells of $\mu$ as shown below:
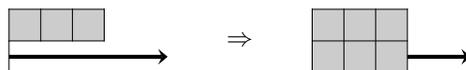
\begin{center}
  \begin{pspicture}(0,0)(5,-2)
    \gcell(1,1)[] \gcell(1,2)[] \gcell(1,3)[]
    \harrow(2,1)[5] \psgrid(0,0)(3,-1) \psline(0,-1)(0,-2)
  \end{pspicture}
\begin{pspicture}(0,0)(4,-2)
   \rput(2,-1){$\Rightarrow$}
  \end{pspicture}
  \begin{pspicture}(0,0)(5,-2)
    \gcell(1,1)[] \gcell(1,2)[] \gcell(1,3)[]
    \gcell(2,1)[] \gcell(2,2)[] \gcell(2,3)[]
    \harrow(2,4)[2] \psgrid(0,0)(3,-2) \psline(0,-1)(0,-2)
  \end{pspicture}
\end{center}
Then the resulting diagram is a half \dkc.. We define $\phi_1(\mu)$ to be this
half \dkc..  It is easy to see that $\phi_1$ is the inverse map of $\psi_1$. Thus
$\psi_1$ is a bijection.  The `moreover' statement is obvious from the
construction of $\psi_1$.
\end{proof}

Now it is easy to prove Proposition~\ref{prop:half_dk}.

\begin{proof}[Proof of Proposition~\ref{prop:half_dk}]
By Lemma~\ref{lem:ov}, it is enough to show the following:
\begin{equation}
  \label{eq:7}
\sum_{\mu\in \OP_k} (-1)^{s(\mu)} q^{|\mu|} = 1.  
\end{equation}

Let $f:\OP_k\to \OP_k$ be the map changing the last part to be overlined if it
is not overlined and vice versa. Then $f$ is a sign-reversing involution on
$\OP_k$ with only one fixed point $\emptyset$. In other words, if
$\mu\ne\emptyset$, then $(-1)^{s(f(\mu))} q^{|f(\mu)|} = - (-1)^{s(\mu)}
q^{|\mu|} $. Thus the sum in \eqref{eq:7} is equal to $(-1)^{s(\emptyset)}
q^{|\emptyset|} =1$.
\end{proof}

\section{\texorpdfstring{Proof of the formula for $q$-secant numbers}{Proof of the formula for q-secant numbers}}
\label{sec:dk_config}

By the fact that $(1-q)\cdot \U=\V$ and Lemmas~\ref{lem1} and \ref{lem2}, the
left-hand side of \eqref{qsec} is equal to
\begin{equation}
  \label{eq:6}
\sum_{p\in\D_n} \wt(p;\U,\U) = \frac{1}{(1-q)^{2n}}
\sum_{k=0}^n \left( \tbinom{2n}{n-k} - \tbinom{2n}{n-k-1} \right)
\sum_{p\in \MD^*_k} \wt(p;\V-\one,\V-\one).
\end{equation}
Thus in order to get \eqref{qsec} it is sufficient to show the following identity: 
\begin{equation}
  \label{eq:12}
\sum_{p\in\MD^*_k} \wt(p;\V-\one,\V-\one) = q^{k(k+1)}\sum_{i=-k}^{k}(-q)^{-i^2}.
\end{equation}
By considering the left-hand side in terms of partitions and arrows as in the
previous section, we are lead to introduce some terminologies as below.

\begin{defn}
  A \emph{\dkpc.}  is a pair $(\lambda,A)$ of a partition
  $\lambda\subset\delta_{k-1}$ and a set $A$ of arrows occupying a whole row or
  a whole column of $\dk/\lambda$ such that no outer corner of $\dk/\lambda$ is
  occupied by two arrows. Here, by an outer corner we mean a cell $c\in
  \dk/\lambda$ such that $\lambda\cup c$ is a partition.  The \emph{length} of
  an arrow is the number of cells it occupies. We denote by $\Dp{k}$ the set of
  \dkpc.s. The \emph{$q$-weight} of a \dkpc. $C=(\lambda,A)$ is defined by
\[\wt_q(C) = (-1)^{|A|} q^{2|\lambda| + \norm{A}},\]
where $\norm{A}$ is the sum of the arrow lengths.
\end{defn}

For example, the $q$-weight of the \dkpc. in Figure~\ref{fig:dk} is $(-1)^7
q^{2\cdot 8 + 1+3+4+3+3+3+2}$.

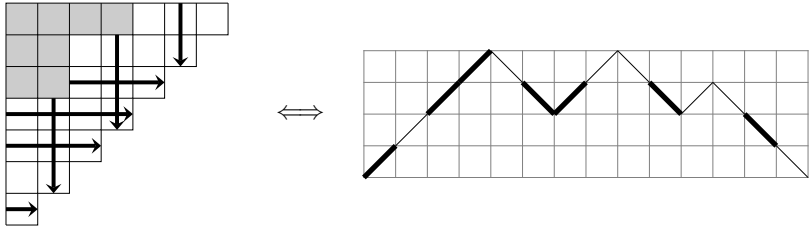
\begin{figure}
 \centering
\begin{pspicture}(0,0)(7,-7)
\gcell(1,1)[]\gcell(1,2)[]\gcell(1,3)[]\gcell(1,4)[]
\gcell(2,1)[]\gcell(2,2)[]
\gcell(3,1)[]\gcell(3,2)[]
\psdk7
\harrow(4,1)[4] \harrow(5,1)[3] \harrow(7,1)[1] \harrow(3,3)[3] 
\varrow(4,2)[3] \varrow(2,4)[3] \varrow(1,6)[2]
\end{pspicture} 
\begin{pspicture}(-4,-1.5)(14,5)
\psgrid[gridcolor=gray](0,0)(14,4)
\rput(-2,2){$\Longleftrightarrow$}
\MUP(0,0)[] \UP(1,1)[$-q^2$] \MUP(2,2)[] \MUP(3,3)[]
\DW(4,4)[$-q^4$] \MDW(5,3)[] \MUP(6,2)[] \UP(7,3)[$-q^4$]
\DW(8,4)[$-q^4$] \MDW(9,3)[] \UP(10,2)[$-q^3$] \DW(11,3)[$-q^3$]
\MDW(12,2)[] \DW(13,1)[$-q$]
\end{pspicture}
\caption{A \dkpc. and the corresponding marked Dyck path, where the marked steps
  are the thicker steps.}
  \label{fig:dk}
\end{figure}

There is a simple bijection between $\MD^*_k$ and $\Dp{k}$ as follows. For
$C=(\lambda,A)\in\Dp{k}$, the north-west border of $\dk/\lambda$ defines a
marked Dyck path of length $2k$ where the marked steps correspond to the
segments on the border with an arrow, see Figure~\ref{fig:dk}. Moreover, if
$p\in \MD^*_k$ corresponds to $C\in\Dp{k}$, one can show that
$\wt(p;\V-1,\V-1) =q^{k(k+1)}\wt_{q^{-1}}(C)$. Thus we obtain the following:

\begin{lem}\label{lem:wt}
For any nonnegative integer $k$, we have 
\[ \sum_{p\in \MD^*_k} \wt(p;\V-1,\V-1) = 
q^{k(k+1)} \sum_{C\in\Dp{k}}\wt_{q^{-1}}(C). \]   
\end{lem}

To complete the bijective proof of \eqref{qsec}, it remains only to show 
that, as announced in the introduction, \dkc.s are a combinatorial model whose weight
sum is $\sum_{i=-k}^{k}(-q)^{-i^2}$.

\begin{thm}\label{thm:dkc} For a nonnegative integer $k$, we have
\[
  \sum_{C\in\Dp{k}}\wt_q(C) = \sum_{i=-k}^{k} (-q)^{i^2}.
\]
Since $\Dp{0}$ only contains $(\emptyset,\emptyset)$ which has weight $1$, this is equivalent to the recurrence:
\begin{equation}
  \sum_{C\in\Dp{k}}\wt_q(C) =  \sum_{C\in\Dp{k-1}}\wt_q(C)  + 2(-q)^{k^2}.  \label{recdelta}
\end{equation}
\end{thm}

We will prove \eqref{recdelta} in the next section.
This identity suggests that $\Dp{k-1}$ can be seen as a subset of $\Dp{k}$
so that the complement has weight sum $2(-q)^{k^2}$.
Actually, we will do this in the other direction, by transforming certain
\dkpc.s into $\delta^+_{k-1}$-configurations.
The key idea of this transformation is ``moving'' the arrows upwards or to the left
to decrease the arrow length by $1$.

\section{The proof of Theorem~\ref{thm:dkc}}

We prove here Theorem~\ref{thm:dkc}, which is the last step of the
bijective proof  of \eqref{qsec}. It states that the generating function of $\Dp{k}$ is the one
of $\Dp{k-1}$ plus $2(-q)^{k^2}$. The general idea is to:
\begin{itemize} 
 \item define a set $\Dm k$ which is in (weight-preserving) bijection with $\Dp k$ and contains $\Dp {k-1}$,
 \item give a sign-reversing involution showing that $\Dm k - \Dp {k-1}$  has
   weight sum $2(-q)^{k^2}$.
\end{itemize}
The proof of Theorem~\ref{thm:dkc} will immediately follow these constructions.

We have defined \dkpc.s before. In this section we introduce more general
objects and redefine \dkpc.s in this general context.

\subsection{Definitions}

\begin{defn}
  A \emph{\dkc.}  is a pair $(\lambda,A)$ of a partition
  $\lambda\subset\delta_{k-1}$ and a set $A$ of arrows occupying a whole row or
  a whole column of $\dk/\lambda$ or $\delta_{k-1}/\lambda$. If the cells
  occupied by an arrow is a row or a column of $\dk/\lambda$
  (resp.~$\delta_{k-1}/\lambda$), we call it a $k$-arrow (resp.~$(k-1)$-arrow).
  The \emph{length} of an arrow is the number of cells occupied by the arrow. An
  \emph{inner corner} (resp.~\emph{outer corner}) is a cell $c\in\lambda$
  (resp.~$c\in\dk/\lambda$) such that $\lambda\setminus\{c\}$
  (resp.~$\lambda\cup\{c\}$) is a partition.  A \emph{forbidden corner} is an
  outer corner which is occupied by two $k$-arrows.  A \emph{\dkpc.} is a
  \dkc. satisfying the following.
  \begin{enumerate}
  \item There is no forbidden corner.
  \item There is no $(k-1)$-arrow.
  \end{enumerate}
We define the \emph{$q$-weight}  of a \dkc. $C=(\lambda,A)$ to be
\[\wt_q(C) = (-1)^{|A|} q^{2|\lambda| + \norm{A}},\]
where $\norm{A}$ is the sum of the arrow lengths.
\end{defn}

Note that a \dkc. may have an arrow (necessarily a $(k-1)$-arrow) of length
0. We will represent such an arrow as a half dot, see the bottom left diagram in
Figure~\ref{fig:psi}. 

We need several terminologies. For a partition $\lambda$, we denote its
transposition by $\lambda^{tr}$. Let $(\lambda,A)$ be a \dkc.. By Row $i$
(resp.~Column $i$), we mean the $i$th uppermost row (resp.~leftmost column).  An
arrow of $(\lambda,A)$ is \emph{ascendible} if the following hold.
\begin{enumerate}
\item It is a $k$-arrow.
\item If it is a horizontal (resp.~vertical) arrow in Row $i$ (resp.~Column
  $i$), then $i\geq2$, $\lambda_{i-1}=\lambda_i$
  (resp.~$\lambda^{tr}_{i-1}=\lambda^{tr}_i$) and there is no arrow in Row $i-1$
  (resp.~Column $i-1$).
\end{enumerate}
By \emph{ascending an (ascendible) arrow}, we mean that we move the arrow one
step upward or to the left as shown below:
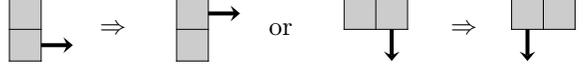
\begin{center}
  \begin{pspicture}(0,0)(3,-2)
\gcell(1,1)[] \gcell(2,1)[] \harrow(2,2)[1] \psgrid(0,0)(1,-2)
  \end{pspicture}
  \begin{pspicture}(-2,0)(3,-2)
\rput(-2,-1){$\Rightarrow$}
\gcell(1,1)[] \gcell(2,1)[] \harrow(1,2)[1] \psgrid(0,0)(1,-2)
  \end{pspicture}
  \begin{pspicture}(-2,0)(3,-2)
\rput(-2,-1){or}
\gcell(1,1)[] \gcell(1,2)[] \varrow(2,2)[1] \psgrid(0,0)(2,-1)
  \end{pspicture}
  \begin{pspicture}(-2,0)(2.2,-2)
\rput(-1.5,-1){$\Rightarrow$}
\gcell(1,1)[] \gcell(1,2)[] \varrow(2,1)[1]\psgrid(0,0)(2,-1)
  \end{pspicture}.
\end{center}

Similarly, an arrow of $(\lambda,A)$ is \emph{descendible} if the following
hold.
\begin{enumerate}
\item It is a $(k-1)$-arrow.
\item If it is a horizontal (resp.~vertical) arrow in Row $i$ (resp.~Column
  $i$), then $i\leq k-1$ and $\lambda_{i}=\lambda_{i+1}$
  (resp.~$\lambda^{tr}_{i}=\lambda^{tr}_{i+1}$), and there is no arrow in Row
  $i+1$ (resp.~Column $i+1$).
\end{enumerate}
By \emph{descending a (descendible) arrow}, we mean that we move the arrow one
step downward or to the right. Thus descending is the inverse operation of
ascending.

We say that a horizontal (resp.~vertical) arrow starts from an (inner or outer)
corner if the arrow and the corner lie in the same row (resp.~column). Note that
an arrow may start from an outer corner and an inner corner simultaneously. We
also say that a corner has an arrow if the arrow starts from the corner. 

An arrow of $(\lambda,A)$ is \emph{shrinkable} if the following hold.
\begin{enumerate}
\item It is a $k$-arrow of length at least $2$.
\item It starts from an outer corner which has only one arrow.
\end{enumerate}
By \emph{shrinking a (shrinkable) arrow} $u$, we mean that we add the outer corner where
$u$ starts to $\lambda$ and decrease the length of $u$ by $2$ as shown below:
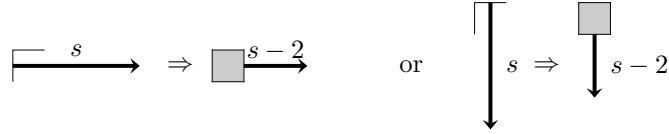
\begin{center}
  \begin{pspicture}(0,0)(4,-2.5)
    \psline(1,0)(0,0)(0,-1)
    \harrow(1,1)[4] \rput(2,0){$s$}
  \end{pspicture}
  \begin{pspicture}(-2,0)(4,-2.5)
    \rput(-1,-.5){$\Rightarrow$}
 \gcell(1,1)[]    \psgrid(0,0)(1,-1)
    \harrow(1,2)[2]    \rput(2,0){$s-2$}
  \end{pspicture}
\begin{pspicture}(-4,0)(1,-4)
    \rput(-2,-2){or}
   \psline(1,0)(0,0)(0,-1)
    \varrow(1,1)[4] \rput[l](1,-2){$s$}
  \end{pspicture}
\begin{pspicture}(-2,0)(3,-4)
    \rput(-1,-2){$\Rightarrow$}
    \gcell(1,1)[] \psgrid(0,0)(1,-1)
    \varrow(2,1)[2] \rput[l](1,-2){$s-2$}
  \end{pspicture}.
\end{center}

An arrow of $(\lambda,A)$ is \emph{stretchable} if the following hold.
\begin{enumerate}
\item It is a $(k-1)$-arrow.
\item It starts from an inner corner which has only one arrow.
\end{enumerate}
By \emph{stretching a (stretchable) arrow} $u$, we mean that we delete the inner corner
where $u$ starts from $\lambda$ and increase the length of $u$ by $2$. Note that
stretching is the inverse operation of shrinking.

An outer corner of $(\lambda,A)$ is \emph{fillable} if it has two arrows of
length at least $1$ and at least one of them is a $(k-1)$-arrow. By
\emph{filling a (fillable) corner}, we mean that we add the outer corner to
$\lambda$ and decrease the lengths of the two arrows by $1$ as shown below:
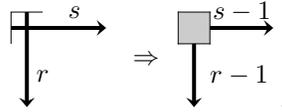
\begin{center}
  \begin{pspicture}(0,0.2)(3,-3)
    \psline(1,0)(0,0)(0,-1)
    \harrow(1,1)[3]     \varrow(1,1)[3] \rput(1,-2){$r$} \rput(2,0){$s$}
  \end{pspicture}
  \begin{pspicture}(-2,0.2)(3.2,-3)
    \rput(-1,-1.5){$\Rightarrow$}
    \gcell(1,1)[] \psgrid(0,0)(1,-1)
    \harrow(1,2)[2]     \varrow(2,1)[2] \rput[l](1,-2){$r-1$} \rput(2,0){$s-1$}
  \end{pspicture}.
\end{center}

An inner corner of $(\lambda,A)$ is \emph{removable} if it has two arrows at
least one of which is a $(k-1)$-arrow. By \emph{removing a (removable) corner},
we mean that we delete the inner corner from $\lambda$ and increase the lengths
of the two arrows by $1$. Note that removing is the inverse operation of
filling.

We define $\Ascend(\lambda,A)$ to be the \dkc. obtained by ascending 
ascendible arrows as many times until there is no ascendible arrow. We define
$\Descend(\lambda,A)$, $\Shrink(\lambda,A)$, $\Stretch(\lambda,A)$,
$\Fill(\lambda,A)$ and $\Remove(\lambda,A)$ in the same way. 
See Figures~\ref{fig:psi} and \ref{fig:phi}.

\begin{figure}
  \centering
\begin{pspicture}(0,0)(8,-8)
\gcell(1,1)[]  \gcell(1,2)[]
\gcell(2,1)[]  \gcell(2,2)[]
\gcell(3,1)[]  \gcell(3,2)[]
\gcell(4,1)[]  \gcell(4,2)[]
\gcell(5,1)[]
\gcell(6,1)[]
\psdk8
\harrow(8,1)[1] \harrow(7,1)[2] \harrow(5,2)[3] \harrow(4,3)[3] \harrow(2,3)[5]
\varrow(1,3)[6] \varrow(1,4)[5] \varrow(1,5)[4] \varrow(1,7)[2]
\end{pspicture}
\begin{pspicture}(0,0)(2,-8)
\rput(0,-5){$\Ascend$}
\rput(0,-6){$\Longrightarrow$}
\end{pspicture}
\begin{pspicture}(0,0)(8,-8)
\gcell(1,1)[]  \gcell(1,2)[]
\gcell(2,1)[]  \gcell(2,2)[]
\gcell(3,1)[]  \gcell(3,2)[]
\gcell(4,1)[]  \gcell(4,2)[]
\gcell(5,1)[]
\gcell(6,1)[]
\psdk8
\harrow(8,1)[1] \harrow(7,1)[2] \harrow(5,2)[3] \harrow(3,3)[3] \harrow(1,3)[5]
\varrow(1,3)[6] \varrow(1,4)[5] \varrow(1,5)[4] \varrow(1,6)[2]
\end{pspicture}
\begin{pspicture}(0,0)(2,-8)
\rput(0,-5){$\Fill$}
\rput(0,-6){$\Longrightarrow$}
\end{pspicture}
\begin{pspicture}(0,0)(8,-8)
\gcell(1,1)[]  \gcell(1,2)[]   \gcell(1,3)[]   \gcell(1,4)[]   \gcell(1,5)[]  \gcell(1,6)[]
\gcell(2,1)[]  \gcell(2,2)[]
\gcell(3,1)[]  \gcell(3,2)[]
\gcell(4,1)[]  \gcell(4,2)[]
\gcell(5,1)[]
\gcell(6,1)[]
\psdk8
\harrow(8,1)[1] \harrow(7,1)[2] \harrow(5,2)[3] \harrow(3,3)[3] \harrow(1,7)[1]
\varrow(2,3)[5] \varrow(2,4)[4] \varrow(2,5)[3] \varrow(2,6)[1]
\end{pspicture}

\begin{pspicture}(0,0)(28,2)
\rput(25,1){$||$} \rput(3,1){$\psi$} \rput(4,1){$\Downarrow$}
\end{pspicture}

\begin{pspicture}(0,0)(8,-8)
\gcell(1,1)[]  \gcell(1,2)[]   \gcell(1,3)[]   \gcell(1,4)[]   \gcell(1,5)[]  \gcell(1,6)[]
\gcell(2,1)[]  \gcell(2,2)[]   \gcell(2,3)[]   \gcell(2,4)[]   \gcell(2,5)[]
\gcell(3,1)[]  \gcell(3,2)[]   \gcell(3,3)[]   \gcell(3,4)[]   \gcell(3,5)[]
\gcell(4,1)[]  \gcell(4,2)[]
\gcell(5,1)[]  \gcell(5,2)[]
\gcell(6,1)[]
\gcell(7,1)[]
\psdk8
\harrow(8,1)[1] \hzero(7,2) \harrow(5,3)[1] \hzero(3,6) \harrow(1,7)[1]
\varrow(4,3)[2] \varrow(4,4)[1] \vzero(4,5) \varrow(2,6)[1]
\end{pspicture}
\begin{pspicture}(0,0)(2,-8)
\rput(0,-5){$\Fill$}
\rput(0,-6){$\Longleftarrow$}
\end{pspicture}
\begin{pspicture}(0,0)(8,-8)
\gcell(1,1)[]  \gcell(1,2)[]   \gcell(1,3)[]   \gcell(1,4)[]   \gcell(1,5)[]  \gcell(1,6)[]
\gcell(2,1)[]  \gcell(2,2)[]   \gcell(2,3)[]   \gcell(2,4)[]   \gcell(2,5)[]
\gcell(3,1)[]  \gcell(3,2)[]
\gcell(4,1)[]  \gcell(4,2)[]
\gcell(5,1)[]  \gcell(5,2)[]
\gcell(6,1)[]
\gcell(7,1)[]
\psdk8
\harrow(8,1)[1] \hzero(7,2) \harrow(5,3)[1] \harrow(3,3)[3] \harrow(1,7)[1]
\varrow(3,3)[3] \varrow(3,4)[2] \varrow(3,5)[1] \varrow(2,6)[1]
\end{pspicture}
\begin{pspicture}(0,0)(2,-8)
\rput(0,-5){$\Shrink$}
\rput(0,-6){$\Longleftarrow$}
\end{pspicture}
\begin{pspicture}(0,0)(8,-8)
\gcell(1,1)[]  \gcell(1,2)[]   \gcell(1,3)[]   \gcell(1,4)[]   \gcell(1,5)[]  \gcell(1,6)[]
\gcell(2,1)[]  \gcell(2,2)[]
\gcell(3,1)[]  \gcell(3,2)[]
\gcell(4,1)[]  \gcell(4,2)[]
\gcell(5,1)[]
\gcell(6,1)[]
\psdk8
\harrow(8,1)[1] \harrow(7,1)[2] \harrow(5,2)[3] \harrow(3,3)[3] \harrow(1,7)[1]
\varrow(2,3)[5] \varrow(2,4)[4] \varrow(2,5)[3] \varrow(2,6)[1]
\end{pspicture}
\caption{The map $\psi$ is the composition
  $\Fill\circ\Shrink\circ\Fill\circ\Ascend$. The upper right configuration is
  equal to the lower right configuration.}
  \label{fig:psi}
\end{figure}
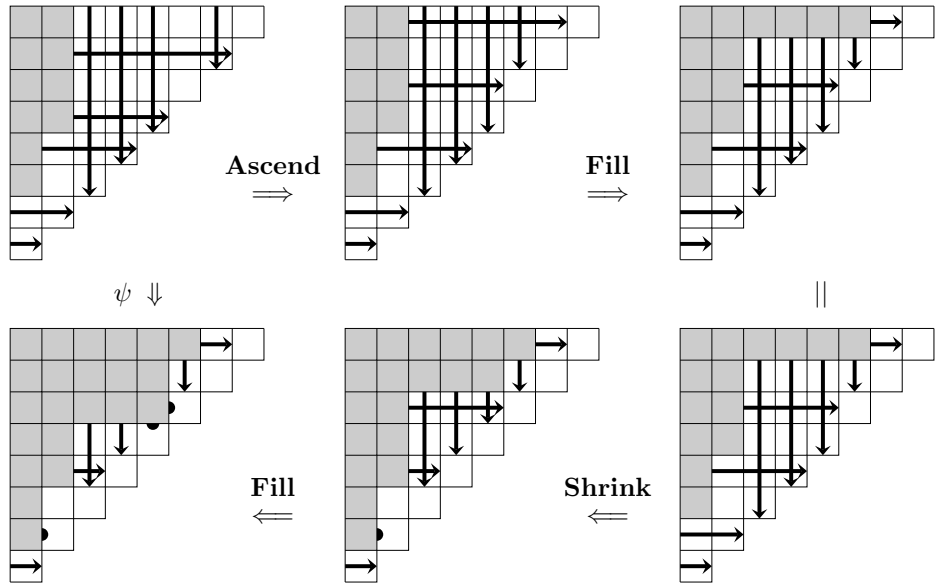

\begin{figure}
  \centering
\begin{pspicture}(0,0)(8,-8)
\gcell(1,1)[]  \gcell(1,2)[]   \gcell(1,3)[]   \gcell(1,4)[]   \gcell(1,5)[]  \gcell(1,6)[]
\gcell(2,1)[]  \gcell(2,2)[]   \gcell(2,3)[]   \gcell(2,4)[]   \gcell(2,5)[]
\gcell(3,1)[]  \gcell(3,2)[]   \gcell(3,3)[]   \gcell(3,4)[]   \gcell(3,5)[]
\gcell(4,1)[]  \gcell(4,2)[]
\gcell(5,1)[]  \gcell(5,2)[]
\gcell(6,1)[]
\gcell(7,1)[]
\psdk8
\harrow(8,1)[1] \hzero(7,2) \harrow(5,3)[1] \hzero(3,6) \harrow(1,7)[1]
\varrow(4,3)[2] \varrow(4,4)[1] \vzero(4,5) \varrow(2,6)[1]
\end{pspicture}
\begin{pspicture}(0,0)(2,-8)
\rput(0,-5){$\Remove$}
\rput(0,-6){$\Longrightarrow$}
\end{pspicture}
\begin{pspicture}(0,0)(8,-8)
\gcell(1,1)[]  \gcell(1,2)[]   \gcell(1,3)[]   \gcell(1,4)[]   \gcell(1,5)[]
\gcell(2,1)[]  \gcell(2,2)[]   \gcell(2,3)[]   \gcell(2,4)[]   \gcell(2,5)[]
\gcell(3,1)[]  \gcell(3,2)[]
\gcell(4,1)[]  \gcell(4,2)[]
\gcell(5,1)[]  \gcell(5,2)[]
\gcell(6,1)[]
\gcell(7,1)[]
\psdk8
\harrow(8,1)[1] \hzero(7,2) \harrow(5,3)[1] \harrow(3,3)[3] \harrow(1,6)[2]
\varrow(3,3)[3] \varrow(3,4)[2] \varrow(3,5)[1] \varrow(1,6)[2]
\end{pspicture}
\begin{pspicture}(0,0)(2,-8)
\rput(0,-5){$\Stretch$}
\rput(0,-6){$\Longrightarrow$}
\end{pspicture}
\begin{pspicture}(0,0)(8,-8)
\gcell(1,1)[]  \gcell(1,2)[]   \gcell(1,3)[]   \gcell(1,4)[]   \gcell(1,5)[]
\gcell(2,1)[]  \gcell(2,2)[]
\gcell(3,1)[]  \gcell(3,2)[]
\gcell(4,1)[]  \gcell(4,2)[]
\gcell(5,1)[]
\gcell(6,1)[]
\psdk8
\harrow(8,1)[1] \harrow(7,1)[2] \harrow(5,2)[3] \harrow(3,3)[3] \harrow(1,6)[2]
\varrow(2,3)[5] \varrow(2,4)[4] \varrow(2,5)[3] \varrow(1,6)[2]
\end{pspicture}

\begin{pspicture}(0,0)(28,2)
\rput(25,1){$||$} \rput(3,1){$\phi$} \rput(4,1){$\Downarrow$}
\end{pspicture}

\begin{pspicture}(0,0)(8,-8)
\gcell(1,1)[]  \gcell(1,2)[]
\gcell(2,1)[]  \gcell(2,2)[]
\gcell(3,1)[]  \gcell(3,2)[]
\gcell(4,1)[]  \gcell(4,2)[]
\gcell(5,1)[]
\gcell(6,1)[]
\psdk8
\harrow(8,1)[1] \harrow(7,1)[2] \harrow(5,2)[3] \harrow(4,3)[3] \harrow(2,3)[5]
\varrow(1,3)[6] \varrow(1,4)[5] \varrow(1,5)[4] \varrow(1,7)[2]
\end{pspicture}
\begin{pspicture}(0,0)(2,-8)
\rput(0,-5){$\Descend$}
\rput(0,-6){$\Longleftarrow$}
\end{pspicture}
\begin{pspicture}(0,0)(8,-8)
\gcell(1,1)[]  \gcell(1,2)[]
\gcell(2,1)[]  \gcell(2,2)[]
\gcell(3,1)[]  \gcell(3,2)[]
\gcell(4,1)[]  \gcell(4,2)[]
\gcell(5,1)[]
\gcell(6,1)[]
\psdk8
\harrow(8,1)[1] \harrow(7,1)[2] \harrow(5,2)[3] \harrow(3,3)[3] \harrow(1,3)[5]
\varrow(1,3)[6] \varrow(1,4)[5] \varrow(1,5)[4] \varrow(1,6)[2]
\end{pspicture}
\begin{pspicture}(0,0)(2,-8)
\rput(0,-5){$\Remove$}
\rput(0,-6){$\Longleftarrow$}
\end{pspicture}
\begin{pspicture}(0,0)(8,-8)
\gcell(1,1)[]  \gcell(1,2)[]   \gcell(1,3)[]   \gcell(1,4)[]   \gcell(1,5)[]
\gcell(2,1)[]  \gcell(2,2)[]
\gcell(3,1)[]  \gcell(3,2)[]
\gcell(4,1)[]  \gcell(4,2)[]
\gcell(5,1)[]
\gcell(6,1)[]
\psdk8
\harrow(8,1)[1] \harrow(7,1)[2] \harrow(5,2)[3] \harrow(3,3)[3] \harrow(1,6)[2]
\varrow(2,3)[5] \varrow(2,4)[4] \varrow(2,5)[3] \varrow(1,6)[2]
\end{pspicture}
\caption{The map $\phi$ is the composition $\Descend \circ \Remove \circ \Stretch \circ \Remove$.
The upper right configuration is equal to the lower right configuration.}
  \label{fig:phi}
\end{figure}
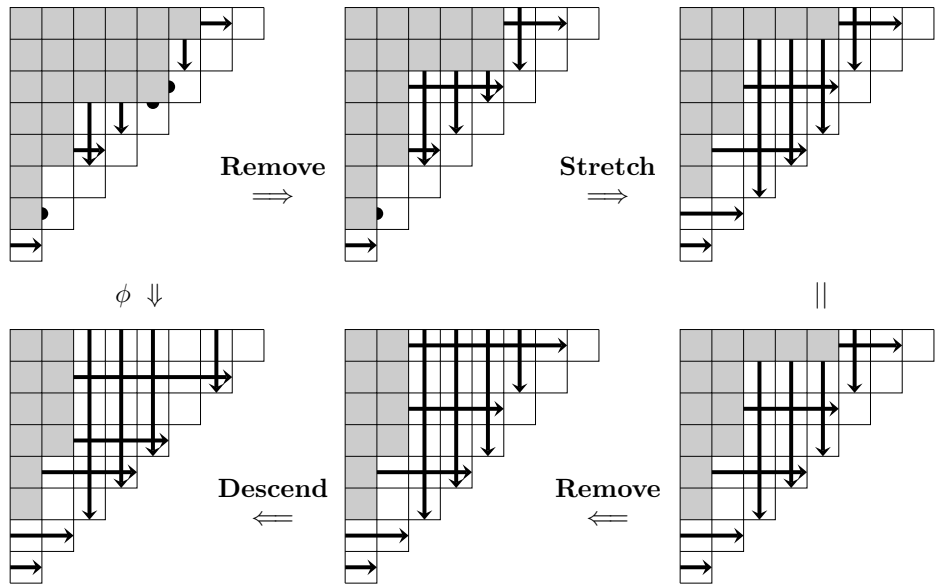

\subsection{Weight-preserving bijections}

We define the map $\psi$ to be the composition
$\Fill\circ\Shrink\circ\Fill\circ\Ascend$, and $\phi$ to be the composition
$\Descend \circ \Remove \circ \Stretch \circ \Remove$. See Figures~\ref{fig:psi}
and \ref{fig:phi}.

A \emph{miniature} of a \dkc. $(\lambda,A)$ is the restriction of $(\lambda,A)$
to the three cells $(k-i,i)$, $(k-i,i+1)$ and $(k-i+1,i)$ for some $1\leq i\leq
k-1$, where any $(k-1)$-arrows in Column $i+1$ or Row $k-i+1$ are ignored (the
cell $(i,j)$ is the cell in Row $i$ and Column $j$).  For example, the
miniatures of
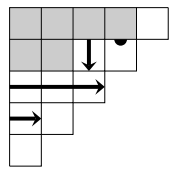
\begin{center}
  \begin{pspicture}(0,0)(5,-5)
    \gcell(1,1)[] \gcell(1,2)[] \gcell(1,3)[] \gcell(1,4)[]
    \gcell(2,1)[] \gcell(2,2)[]
    \psdk5 \vzero(2,4) \varrow(2,3)[1] \harrow(3,1)[3] \harrow(4,1)[1]
  \end{pspicture}
\end{center}
are the following (the bottom miniature is drawn to the left):
\begin{center}
  \begin{pspicture}(0,0)(2,-2)
\psdk2 \harrow(1,1)[1]
 \end{pspicture}, \qquad
 \begin{pspicture}(0,0)(2,-2)
\psdk2 \harrow(1,1)[2]
  \end{pspicture}, \qquad
  \begin{pspicture}(0,0)(2,-2)
\psdk2 \varrow(1,1)[1] \harrow(2,1)[1]
  \end{pspicture}, \qquad
\begin{pspicture}(0,0)(2,-2)
    \gcell(1,1)[]
\psdk2 \vzero(2,1)
  \end{pspicture}.
\end{center}

\begin{defn}\label{def:dm}
  A \emph{\dkmc.} $(\lambda,A)$ is a \dkc. satisfying the following.
  \begin{enumerate}
  \item There is neither fillable corner nor forbidden corner.
  \item Neither Row $k$ nor Column $k$ has an arrow of length $0$.
 \item Each $k$-arrow is of length $1$.
 \item For any miniature $M$, if there is a horizontal (resp.~vertical)
   $k$-arrow in the bottom (resp.~right) cell, then the middle cell is contained
   in $\lambda$. Moreover, if the bottom (resp.~left) cell has a horizontal
   (resp.~vertical) $k$-arrow and a vertical (resp.~horizontal) $(k-1)$-arrow,
   then the left (resp.~bottom) cell has a horizontal (resp.~vertical)
   $k$-arrow. Pictorially, these mean the following:
    
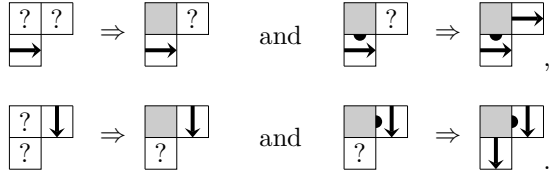
\begin{center}
   \begin{pspicture}(0,0)(2,-2)
       \cell(1,1)[$?$] \cell(1,2)[$?$] \cell(2,1)[] \psline(2,0)(0,0)(0,-2) \harrow(2,1)[1]
    \end{pspicture}
   \begin{pspicture}(-2,0)(2,-2)
\rput(-1,-1){$\Rightarrow$}
     \Gcell(1,1)[] \cell(1,2)[$?$] \cell(2,1)[]  \psline(2,0)(0,0)(0,-2) \harrow(2,1)[1]
    \end{pspicture}
   \begin{pspicture}(-4,0)(2,-2)
\rput(-2,-1){and}
       \Gcell(1,1)[] \cell(1,2)[$?$] \cell(2,1)[] \psline(2,0)(0,0)(0,-2)
      \harrow(2,1)[1] \vzero(2,1)
   \end{pspicture}
   \begin{pspicture}(-2,0)(2,-2)
\rput(-1,-1){$\Rightarrow$}
      \Gcell(1,1)[] \cell(1,2)[] \cell(2,1)[]  \psline(2,0)(0,0)(0,-2)
      \harrow(2,1)[1] \vzero(2,1) \harrow(1,2)[1]
    \end{pspicture},

   \begin{pspicture}(0,0)(2,-2)
      \cell(1,1)[$?$] \cell(1,2)[] \cell(2,1)[$?$] \psline(2,0)(0,0)(0,-2) \varrow(1,2)[1]
    \end{pspicture}
   \begin{pspicture}(-2,0)(2,-2)
\rput(-1,-1){$\Rightarrow$}
       \Gcell(1,1)[] \cell(1,2)[] \cell(2,1)[$?$] \psline(2,0)(0,0)(0,-2) \varrow(1,2)[1]
    \end{pspicture}
   \begin{pspicture}(-4,0)(2,-2)
\rput(-2,-1){and}
      \Gcell(1,1)[] \cell(1,2)[] \cell(2,1)[$?$] \psline(2,0)(0,0)(0,-2)
      \varrow(1,2)[1] \hzero(1,2)
   \end{pspicture}
   \begin{pspicture}(-2,1)(2,-2)
\rput(-1,-1){$\Rightarrow$}
      \Gcell(1,1)[] \cell(1,2)[] \cell(2,1)[] \psline(2,0)(0,0)(0,-2)
      \varrow(1,2)[1] \hzero(1,2) \varrow(2,1)[1]
    \end{pspicture}.
   \end{center}
 \end{enumerate}
We denote by $\Dm{k}$ the set of \dkmc.s.
\end{defn}

\begin{prop}\label{prop:wt-bij}
  The map $\psi$ is a weight-preserving bijection from $\Dp k$ to $\Dm k$. The
  inverse map is $\phi$.
\end{prop}

Before proving Proposition~\ref{prop:wt-bij} in the next subsection, let us
show that it permits to complete the proof of Theorem~\ref{thm:dkc}, following the scheme
given at the beginning of this section.

Note that we can consider $C\in\Dp{k-1}$ as a \dkc. with only $(k-1)$-arrows.
Using this identification we can easily see that $\Dp{k-1}\subset\Dm{k}$.

\begin{lem}\label{lem:exception}
  We have $\Dp{k-1}\subset\Dm{k}$ and the elements in $\Dm{k}\setminus\Dp{k-1}$
  are exactly those which contain at least one miniature belonging to the list
  in Figure~\ref{fig:list}.
\end{lem}
\begin{proof}
  Let $(\lambda,A)\in \Dm{k}\setminus\Dp{k-1}$. Then at least one of the
  following is true: (1) $\lambda$ has the cell $(k-i,i)$ for some $1\leq i\leq
  k-1$, (2) there is a $k$-arrow, or (3) there is an arrow of length $0$. Since
  each of (2) and (3) implies (1), we always have (1). By definition of $\Dm k$,
  it is straightforward to check that the miniature containing the cell
  $(k-i,i)$ must be in the list in Figure~\ref{fig:list}.
\end{proof}

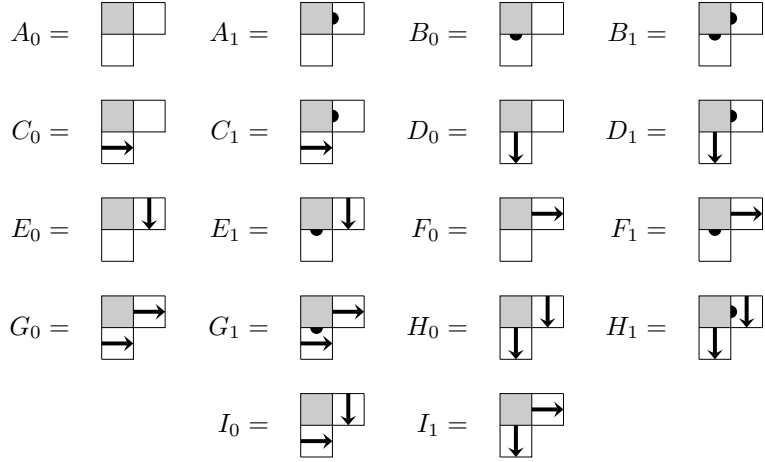
\begin{figure}
\begin{center}
\begin{pspicture}(-2,0)(4,-3)
\rput[r](-1,-1){$A_0=$}
\Gcell(1,1)[] \cell(2,1)[] \cell(1,2)[] \psline(2,0)(0,0)(0,-2)
\end{pspicture}
\begin{pspicture}(-2,0)(4,-3)
\rput[r](-1,-1){$A_1=$}
\Gcell(1,1)[] \cell(2,1)[] \cell(1,2)[] \psline(2,0)(0,0)(0,-2)
\hzero(1,2)
\end{pspicture}
\begin{pspicture}(-2,0)(4,-3)
\rput[r](-1,-1){$B_0=$}
\Gcell(1,1)[] \cell(2,1)[] \cell(1,2)[] \psline(2,0)(0,0)(0,-2)
\vzero(2,1)
\end{pspicture}
\begin{pspicture}(-2,0)(4,-3)
\rput[r](-1,-1){$B_1=$}
\Gcell(1,1)[] \cell(2,1)[] \cell(1,2)[] \psline(2,0)(0,0)(0,-2)
\hzero(1,2)
\vzero(2,1)
\end{pspicture}

\begin{pspicture}(-2,0)(4,-3)
\rput[r](-1,-1){$C_0=$}
\Gcell(1,1)[] \cell(2,1)[] \cell(1,2)[] \psline(2,0)(0,0)(0,-2)
\harrow(2,1)[1]
\end{pspicture}
\begin{pspicture}(-2,0)(4,-3)
\rput[r](-1,-1){$C_1=$}
\Gcell(1,1)[] \cell(2,1)[] \cell(1,2)[] \psline(2,0)(0,0)(0,-2)
\harrow(2,1)[1]
\hzero(1,2)
\end{pspicture}
\begin{pspicture}(-2,0)(4,-3)
\rput[r](-1,-1){$D_0=$}
\Gcell(1,1)[] \cell(2,1)[] \cell(1,2)[] \psline(2,0)(0,0)(0,-2)
\varrow(2,1)[1]
\end{pspicture}
\begin{pspicture}(-2,0)(4,-3)
\rput[r](-1,-1){$D_1=$}
\Gcell(1,1)[] \cell(2,1)[] \cell(1,2)[] \psline(2,0)(0,0)(0,-2)
\varrow(2,1)[1]
\hzero(1,2)
\end{pspicture}

\begin{pspicture}(-2,0)(4,-3)
\rput[r](-1,-1){$E_0=$}
\Gcell(1,1)[] \cell(2,1)[] \cell(1,2)[] \psline(2,0)(0,0)(0,-2)
\varrow(1,2)[1]
\end{pspicture}
\begin{pspicture}(-2,0)(4,-3)
\rput[r](-1,-1){$E_1=$}
\Gcell(1,1)[] \cell(2,1)[] \cell(1,2)[] \psline(2,0)(0,0)(0,-2)
\vzero(2,1)
\varrow(1,2)[1]
\end{pspicture}
\begin{pspicture}(-2,0)(4,-3)
\rput[r](-1,-1){$F_0=$}
\Gcell(1,1)[] \cell(2,1)[] \cell(1,2)[] \psline(2,0)(0,0)(0,-2)
\harrow(1,2)[1]
\end{pspicture}
\begin{pspicture}(-2,0)(4,-3)
\rput[r](-1,-1){$F_1=$}
\Gcell(1,1)[] \cell(2,1)[] \cell(1,2)[] \psline(2,0)(0,0)(0,-2)
\harrow(1,2)[1]
\vzero(2,1)
\end{pspicture}

\begin{pspicture}(-2,0)(4,-3)
\rput[r](-1,-1){$G_0=$}
\Gcell(1,1)[] \cell(2,1)[] \cell(1,2)[] \psline(2,0)(0,0)(0,-2)
\harrow(2,1)[1]\harrow(1,2)[1]
\end{pspicture}
\begin{pspicture}(-2,0)(4,-3)
\rput[r](-1,-1){$G_1=$}
\Gcell(1,1)[] \cell(2,1)[] \cell(1,2)[] \psline(2,0)(0,0)(0,-2)
\harrow(2,1)[1]\harrow(1,2)[1]  \vzero(2,1)
\end{pspicture}
\begin{pspicture}(-2,0)(4,-3)
\rput[r](-1,-1){$H_0=$}
\Gcell(1,1)[] \cell(2,1)[] \cell(1,2)[] \psline(2,0)(0,0)(0,-2)
\varrow(2,1)[1]\varrow(1,2)[1]
\end{pspicture}
\begin{pspicture}(-2,0)(4,-3)
\rput[r](-1,-1){$H_1=$}
\Gcell(1,1)[] \cell(2,1)[] \cell(1,2)[] \psline(2,0)(0,0)(0,-2)
\varrow(2,1)[1]\varrow(1,2)[1]  \hzero(1,2)
\end{pspicture}

\begin{pspicture}(-2,0)(4,-3)
\rput[r](-1,-1){$I_0=$}
\Gcell(1,1)[] \cell(2,1)[] \cell(1,2)[] \psline(2,0)(0,0)(0,-2)
\harrow(2,1)[1]\varrow(1,2)[1]
\end{pspicture}
\begin{pspicture}(-2,0)(4,-3)
\rput[r](-1,-1){$I_1=$}
\Gcell(1,1)[] \cell(2,1)[] \cell(1,2)[] \psline(2,0)(0,0)(0,-2)
\varrow(2,1)[1]\harrow(1,2)[1]
\end{pspicture}
\end{center}
 \caption{List of exceptions.}
  \label{fig:list}
\end{figure}

Now we define a sign-reversing involution on $\Dm k\setminus\Dp{k-1}$. For an
element $C$ in $\Dm k\setminus\Dp{k-1}$, we find the uppermost miniature of $C$
that is contained in the list in Figure~\ref{fig:list} except $I_0$ and
$I_1$. Such miniature exists unless all the miniatures of $C$ are either $I_0$
or $I_1$, and there are exactly two possibilities for this, see the
Figure~\ref{fig:fixed} when $k=5$. If this miniature is $X_0$ (resp.~$X_1$) then
we define $f(C)$ to be the \dkc. obtained from $C$ by replacing the miniature
with $X_1$ (resp.~$X_0$), where $X$ is one of $A,B,\dots,H$, see
Figure~\ref{fig:sign-reversing}. This gives:

\begin{lem}
The map $f$ is a sign-reversing involution on $\Dm k\setminus\Dp{k-1}$ with two fixed
points of weight $(-q)^{k^2}$.
\end{lem}

This ends the proof of Theorem~\ref{thm:dkc} (upon completion of the proof in the next subsection). 

\begin{figure}
 \centering
\begin{pspicture}(0,0)(5,-5)
\gcell(1,1)[]\gcell(1,2)[]\gcell(1,3)[]\gcell(1,4)[]
\gcell(2,1)[]\gcell(2,2)[]\gcell(2,3)[]
\gcell(3,1)[]\gcell(3,2)[]
\gcell(4,1)[]
\psdk 5
\harrow(5,1)[1]
\varrow(4,2)[1]
\harrow(3,3)[1]
\varrow(2,4)[1]
\harrow(1,5)[1]
\end{pspicture} \qquad \qquad
\begin{pspicture}(0,0)(5,-5)
\gcell(1,1)[]\gcell(1,2)[]\gcell(1,3)[]\gcell(1,4)[]
\gcell(2,1)[]\gcell(2,2)[]\gcell(2,3)[]
\gcell(3,1)[]\gcell(3,2)[]
\gcell(4,1)[]
\psdk 5
\varrow(5,1)[1]
\harrow(4,2)[1]
\varrow(3,3)[1]
\harrow(2,4)[1]
\varrow(1,5)[1]
\end{pspicture}
\caption{The two fixed points in $\Dp5$ which will not be sent to $\Dp4$.}
  \label{fig:fixed}
\end{figure}
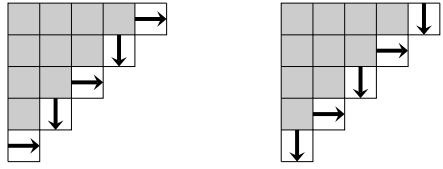

 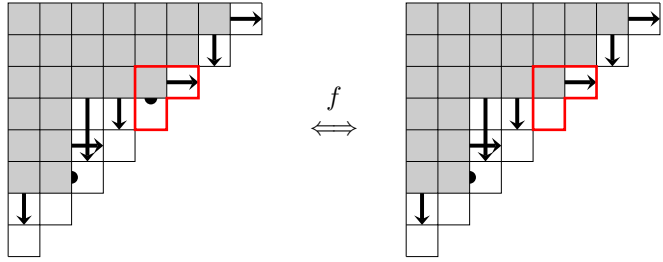
\begin{figure}
   \centering
\begin{pspicture}(0,0)(8,-8)
\gcell(1,1)[]  \gcell(1,2)[] \gcell(1,3)[] \gcell(1,4)[] \gcell(1,5)[]\gcell(1,6)[]\gcell(1,7)[]
\gcell(2,1)[]  \gcell(2,2)[] \gcell(2,3)[] \gcell(2,4)[] \gcell(2,5)[]\gcell(2,6)[]
\gcell(3,1)[]  \gcell(3,2)[]   \gcell(3,3)[]   \gcell(3,4)[]   \gcell(3,5)[]
\gcell(4,1)[]  \gcell(4,2)[]
\gcell(5,1)[]  \gcell(5,2)[]
\gcell(6,1)[]  \gcell(6,2)[]
\psdk8
\harrow(1,8)[1] \varrow(2,7)[1] \harrow(3,6)[1] \vzero(4,5)
\varrow(4,4)[1] \varrow(4,3)[2] \harrow(5,3)[1] \hzero(6,3)
\varrow(7,1)[1]
\pspolygon[linewidth=1pt, linecolor=red](4,-2)(4,-4)(5,-4)(5,-3)(6,-3)(6,-2)
\end{pspicture}
\begin{pspicture}(0,0)(4,-8)
\rput(2,-3){$f$}
\rput(2,-4){$\Longleftrightarrow$}
\end{pspicture}
\begin{pspicture}(0,0)(8,-8)
\gcell(1,1)[]  \gcell(1,2)[] \gcell(1,3)[] \gcell(1,4)[] \gcell(1,5)[]\gcell(1,6)[]\gcell(1,7)[]
\gcell(2,1)[]  \gcell(2,2)[] \gcell(2,3)[] \gcell(2,4)[] \gcell(2,5)[]\gcell(2,6)[]
\gcell(3,1)[]  \gcell(3,2)[]   \gcell(3,3)[]   \gcell(3,4)[]   \gcell(3,5)[]
\gcell(4,1)[]  \gcell(4,2)[]
\gcell(5,1)[]  \gcell(5,2)[]
\gcell(6,1)[]  \gcell(6,2)[]
\psdk8
\harrow(1,8)[1] \varrow(2,7)[1] \harrow(3,6)[1]
\varrow(4,4)[1] \varrow(4,3)[2] \harrow(5,3)[1] \hzero(6,3)
\varrow(7,1)[1]
\pspolygon[linewidth=1pt,linecolor=red](4,-2)(4,-4)(5,-4)(5,-3)(6,-3)(6,-2)
\end{pspicture}
\caption{The sign-reversing involution $f$ on $\Dm k\setminus\Dp{k-1}$. The
  uppermost miniature which is in Figure~\ref{fig:list} except $I_0$ and $I_1$
  is colored red.}
   \label{fig:sign-reversing}
 \end{figure}

\subsection{Proof of Proposition~\ref{prop:wt-bij}}

In order to prove Proposition~\ref{prop:wt-bij} we need several lemmas.

\begin{lem}\label{lem:no ascendible}
  Let $(\lambda,A)\in\Dp k$. Then we cannot make an ascendible arrow by
  repeating the operations filling and shrinking on
  $\Ascend(\lambda,A)$.
\end{lem}
\begin{proof}
  Assume that we can make an ascendible arrow, say $u$, after $r$ operations,
  where $r$ is the smallest such integer. By symmetry we can assume that $u$ is
  horizontal. Since we can only add cells to $\lambda$, the last operation must
  add the cell $c$ to which the arrow $u$ is attached as shown below:
  
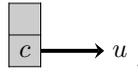
\begin{center}
\begin{pspicture}(0,0)(4,-2)
\gcell(1,1)[] \gcell(2,1)[] \harrow(2,2)[2] \psgrid(0,0)(1,-2)
\rput(3.5,-1.5){$u$}
\rput(.5,-1.5){$c$}
  \end{pspicture}.
 \end{center}
Since $u$ is ascendible, it is a $k$-arrow. Thus $u$ cannot be shrinked, which
implies $c$ is added by filling operation. Before filling $c$, we have the
following situation:
  
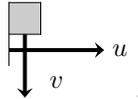
\begin{center}
\begin{pspicture}(0,0)(4,-3)
\gcell(1,1)[] \harrow(2,1)[3] \psgrid(0,0)(1,-1)
\psline(0,-1)(0,-2)
\varrow(2,1)[2]
\rput(3.5,-1.5){$u$}
\rput(1.5,-2.5){$v$}
  \end{pspicture}.
 \end{center}
 Since $c$ is a fillable cell, $v$ is a $(k-1)$-arrow. Since $v$ cannot be
 obtained by ascending, it has become a $(k-1)$-arrow by shrinking. Thus before
 shrinking $v$, we have the following situation:
  
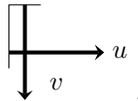
\begin{center}
\begin{pspicture}(0,0)(4,-3)
\psline(1,0)(0,0)(0,-2) \harrow(2,1)[3]
\varrow(1,1)[3]
\rput(3.5,-1.5){$u$}
\rput(1.5,-2.5){$v$}
  \end{pspicture}.
 \end{center}
Thus we can make $u$ an ascendible arrow using $r-2$ operations which is a
contradiction to the minimal condition on $r$.
\end{proof}

If a \dkc. has the following then it is called \emph{invalid}:
  
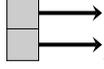
\begin{center}
 \begin{pspicture}(0,0)(3,-2)
\gcell(1,1)[] \gcell(2,1)[] \harrow(2,2)[2] \harrow(1,2)[2]
\psgrid(0,0)(1,-2)
  \end{pspicture}.
\end{center}

\begin{lem}\label{lem:valid}
  Let $(\lambda,A)\in\Dp k$. Then we cannot make an invalid \dkc. by repeating
  the operations filling, removing, ascending, and shrinking on
  $\Ascend(\lambda,A)$.
\end{lem}
\begin{proof}
  Assume that we have the following:
  
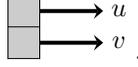
\begin{center}
 \begin{pspicture}(0,0)(4,-2)
\gcell(1,1)[] \gcell(2,1)[] \harrow(2,2)[2]
\harrow(1,2)[2] \psgrid(0,0)(1,-2)
\rput(3.5,-1.5){$v$}
\rput(3.5,-.5){$u$}
 \end{pspicture}.
\end{center}
Then $u$ is a $(k-1)$-arrow and $v$ is a $k$-arrow. Since $u$ cannot become a
$(k-1)$-arrow by ascending, it has been shrinked. However, after shrinking $u$,
we have the following situation
  \begin{center}
 \begin{pspicture}(0,0)(3,-1)
\gcell(1,1)[] \harrow(1,2)[2] \psgrid(0,0)(1,-1)
\rput(3.5,-.5){$u$} \rput(.5,-.5){$c$}
 \end{pspicture}
\end{center}
where the cell below $c$ cannot be added. Thus the \dkc. cannot be invalid.
\end{proof}

\begin{lem}\label{lem:psi}
  The map $\psi$ is a function from $\Dp k$ to $\Dm k$.
\end{lem}
\begin{proof}
  Let $(\lambda,A)\in\Dp k$. We will show that $\psi (\lambda,A)$ satisfies the
  three conditions (1), (2), and (3) in Definition~\ref{def:dm}.

  Condition (1) follows from the fact that filling, ascending, and shrinking
  cannot produce a forbidden corner and that $\psi$ ends with $\Fill$, which
  leaves no fillable corner.

  To prove Condition (2), assume that an arrow $u$ in $(\lambda,A)$ remains a
  $k$-arrow of length at least $2$ in $\psi(\lambda,A)$. We can assume that $u$
  is horizontal and the uppermost such arrow. Now consider the arrow $u$ in
  $\Fill\circ\Ascend(\lambda,A)$. If $u$ starts from a non-fillable outer conner
  as shown below,
\begin{center}
  \begin{pspicture}(0,0)(4.5,-1) \psline(1,0)(0,0)(0,-1) \harrow(1,1)[4]
    \rput(4.5,-.5){$u$}
  \end{pspicture}
\end{center}
then $u$ will be a $(k-1)$-arrow by shrinking. Otherwise, we have one of the
following situations:
 
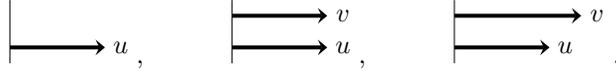
\begin{center}
\begin{pspicture}(0,0)(4,-2)
\psline(0,0)(0,-2) \harrow(2,1)[3]
\rput(3.5,-1.5){$u$}
\end{pspicture}, \qquad\quad
\begin{pspicture}(0,0)(4,-2)
\psline(0,0)(0,-2) \harrow(2,1)[3]
\harrow(1,1)[3]
\rput(3.5,-1.5){$u$}
\rput(3.5,-.5){$v$}
 \end{pspicture}, \qquad\quad
\begin{pspicture}(0,0)(5,-2)
\psline(0,0)(0,-2) \harrow(2,1)[3]
\harrow(1,1)[4]
\rput(3.5,-1.5){$u$}
\rput(4.5,-.5){$v$}
 \end{pspicture}.
\end{center}
In the left situation above $u$ is ascendible, which is a contradiction to
Lemma~\ref{lem:no ascendible}. In the middle situation, we have an invalid
\dkc., which is a contradiction to Lemma~\ref{lem:valid}. In the right
situation, $v$ is a $k$-arrow. Since $u$ is the uppermost $k$-arrow of length at
least $2$ in $\psi(\lambda,A)$, $v$ must become a $(k-1)$-arrow. This is only
possible by shrinking. However, after shrinking $v$, we can also shrink
$u$. Thus $u$ will become a $(k-1)$-arrow, which is a contradiction.  Thus we
cannot have such $u$ in all cases and Condition (2) is true.

It remains to prove Condition (3) is held. By symmetry it is sufficient to prove
the following:
   
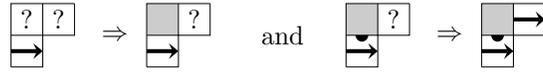
\begin{center}
   \begin{pspicture}(0,0)(2,-2)
       \cell(1,1)[$?$] \cell(1,2)[$?$] \cell(2,1)[] \psline(2,0)(0,0)(0,-2) \harrow(2,1)[1]
    \end{pspicture}
   \begin{pspicture}(-2,0)(2,-2)
\rput(-1,-1){$\Rightarrow$}
     \Gcell(1,1)[] \cell(1,2)[$?$] \cell(2,1)[]  \psline(2,0)(0,0)(0,-2) \harrow(2,1)[1]
    \end{pspicture}
   \begin{pspicture}(-4,0)(2,-2)
\rput(-2,-1){and}
       \Gcell(1,1)[] \cell(1,2)[$?$] \cell(2,1)[] \psline(2,0)(0,0)(0,-2)
      \harrow(2,1)[1] \vzero(2,1)
   \end{pspicture}
   \begin{pspicture}(-2,0)(2,-2)
\rput(-1,-1){$\Rightarrow$}
      \Gcell(1,1)[] \cell(1,2)[] \cell(2,1)[]  \psline(2,0)(0,0)(0,-2)
      \harrow(2,1)[1] \vzero(2,1) \harrow(1,2)[1]
    \end{pspicture}
  \end{center}

If a horizontal $k$-arrow $u$, which is necessarily of length $1$, does not
start from an outer corner, we have one of the following situations:
   
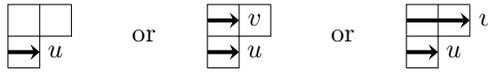
\begin{center}
  \begin{pspicture}(0,0)(2,-2)
    \cell(1,1)[] \cell(1,2)[] \cell(2,1)[]  \psline(2,0)(0,0)(0,-2) \harrow(2,1)[1]
\rput(1.5,-1.5){$u$}
    \end{pspicture}
   \begin{pspicture}(-4,0)(2,-2)
\rput(-2,-1){or}
     \cell(1,1)[] \cell(1,2)[] \cell(2,1)[]  \psline(2,0)(0,0)(0,-2)
     \harrow(2,1)[1]     \harrow(1,1)[1]
\rput(1.5,-1.5){$u$} \rput(1.5,-.5){$v$}
  \end{pspicture}
   \begin{pspicture}(-4,0)(2,-2)
\rput(-2,-1){or}
     \cell(1,1)[] \cell(1,2)[] \cell(2,1)[]  \psline(2,0)(0,0)(0,-2)
     \harrow(2,1)[1] \harrow(1,1)[2]
\rput(1.5,-1.5){$u$} \rput(2.5,-.5){$v$}
  \end{pspicture}
 \end{center}
In the left situation above $u$ is ascendible, which is a contradiction to
Lemma~\ref{lem:no ascendible}. The middle situation contradicts to
Lemma~\ref{lem:valid}. The right situation contradicts to Condition (2).

Now assume that a horizontal $k$-arrow $u$ starts from an outer corner with
another arrow, say $v$, which is a $(k-1)$-arrow and of length $0$. If Condition
(3) is not true, we have one of the following situations:
   
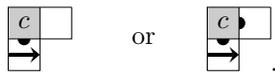
\begin{center}
  \begin{pspicture}(0,0)(2,-2)
      \Gcell(1,1)[] \cell(1,2)[] \cell(2,1)[]  \psline(2,0)(0,0)(0,-2)
      \harrow(2,1)[1] \vzero(2,1) \rput(.5,-.5){$c$}
   \end{pspicture}
   \begin{pspicture}(-4,0)(2,-2)
\rput(-2,-1){or}
      \Gcell(1,1)[] \cell(1,2)[] \cell(2,1)[]  \psline(2,0)(0,0)(0,-2)
      \harrow(2,1)[1] \vzero(2,1) \hzero(1,2) \rput(.5,-.5){$c$}
 \end{pspicture}.
\end{center}
Let $c$ be the cell above $u$. In the left situation above, $c$ is constructed by
shrinking $v$. If we omit this shrinking operation, then we get the following:
   
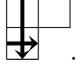
\begin{center}
   \begin{pspicture}(0,0)(2,-2)
    \cell(1,1)[] \cell(1,2)[] \cell(2,1)[]  \psline(2,0)(0,0)(0,-2)
     \harrow(2,1)[1] \varrow(1,1)[2]
  \end{pspicture}.
\end{center}
This means that we can make an ascendible arrow by filling and shrinking from
$\Ascend(\lambda,A)$. This is a contradiction to Lemma~\ref{lem:no ascendible}.
Similarly, if we have the right situation above, then $c$ is constructed by
filling operation. Thus before filling $c$, we have
   
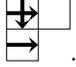
\begin{center}
   \begin{pspicture}(0,0)(2,-2)
    \cell(1,1)[] \cell(1,2)[] \cell(2,1)[]  \psline(2,0)(0,0)(0,-2)
     \harrow(2,1)[1] \varrow(1,1)[1] \harrow(1,1)[1]
  \end{pspicture}.
\end{center}
This is a contradiction to Lemma~\ref{lem:valid}. Thus Condition (3) is true.
\end{proof}

\begin{lem}\label{lem:phi}
  The map $\phi$ is a function from $\Dm k$ to $\Dp k$.
\end{lem}
\begin{proof}
  Let $(\lambda,A)\in \Dm k$. We have to show that there is neither forbidden
  cell nor $(k-1)$-arrow in $\phi(\lambda,A)$. Since filling, removing,
  stretching, and descending cannot produce a forbidden corner, there is no
  forbidden corner in $\phi(\lambda,A)$. Suppose that there is a $(k-1)$-arrow
  $u$ in $(\lambda,A)$ which remains a $(k-1)$-arrow in $\phi(\lambda,A)$. By
  symmetry we can assume that $u$ is horizontal and it is the lowest such
  arrow. We now consider $\Remove(\lambda,A)$. If $u$ starts from an inner
  corner of $\Remove(\lambda,A)$, then it will be stretched and become a
  $k$-arrow.  Otherwise, we have one of the following situations:
 
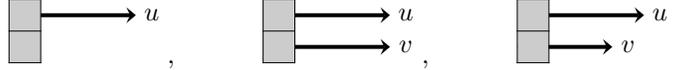
\begin{center}
\begin{pspicture}(0,0)(5,-2)
\harrow(1,2)[3]
\rput(4.5,-.5){$u$}
\gcell(1,1)[] \gcell(2,1)[] \psgrid(0,0)(1,-2)
\end{pspicture}, \qquad\quad
\begin{pspicture}(0,0)(5,-2)
\harrow(1,2)[3]
\rput(4.5,-.5){$u$}
\harrow(2,2)[3]
\rput(4.5,-1.5){$v$}
\gcell(1,1)[] \gcell(2,1)[] \psgrid(0,0)(1,-2)
\end{pspicture}, \qquad\quad
\begin{pspicture}(0,0)(5,-2)
\harrow(1,2)[3]
\rput(4.5,-.5){$u$}
\harrow(2,1)[3]
\rput(3.5,-1.5){$v$}
\gcell(1,1)[] \gcell(2,1)[] \psgrid(0,0)(1,-2)
\end{pspicture}. 
\end{center}
In the left situation, $u$ will be descended. The middle situation is impossible
by Lemma~\ref{lem:valid}. In the right situation, because of the minimality of
$u$, we know that $v$ becomes a $(k-1)$-arrow. Thus $v$ will be either stretched
or descended. It is easy to see that $u$ must be also stretched or
descended. Thus there is no $(k-1)$-arrow in $\phi(\lambda,A)$.
\end{proof}

By Lemmas~\ref{lem:psi} and \ref{lem:phi} and the fact that $\phi$ is the
inverse operation of $\psi$, we obtain Proposition~\ref{prop:wt-bij}.

\section{\texorpdfstring{Limiting case $k\to\infty$ and connection with the
    triple product identity}{ Limiting case k→∞ and connection with the triple
    product identity}}
\label{sec:limit}

For $C=(\lambda,A)\in \Dp{k}$, if there is an arrow coming from an (nonempty)
row or column of $\lambda$, one can easily see that $q^k$ divides $\wt_q(C)$.
In other words, if $\wt_q(C)$ is not divisible by $q^k$, then the partition, the
horizontal arrows and the vertical arrows are completely separated. Thus we can
freely choose a partition, vertical arrows of distinct lengths and horizontal
arrows of distinct lengths. This argument gives us the following.
\begin{prop}\label{thm:separate}
For any nonnegative integer $k$, we have
\[ \sum_{C\in\Dp{k}}\wt_q(C) \equiv \prod_{i\geq1} \frac{1}{1-q^{2i}}
\prod_{i\geq1} ({1-q^{i}})\prod_{i\geq1} ({1-q^{i}}) 
=\prod_{i\geq1} \frac{1-q^{i} }{1+q^{i}} \mod q^k.\]
\end{prop}

Letting $k\to\infty$ in Theorem~\ref{thm:dkc} and Proposition~\ref{thm:separate}, we get
\eqref{eq:gauss}. 

\subsection{A connection with the triple product identity}

Now we define the \emph{$(y,q)$-weight} of a \dkpc. $C=(\lambda,A)$ to be
\[\wt_{y,q}(C) = (-1)^{|A|} q^{2|\lambda| + \norm{A}} (-y)^{oh(A)-ov(A)},\]
where $oh(A)$ (resp.~$ov(A)$) is the number of odd-length horizontal (resp.~vertical)
arrows in $A$. For example, the $(y,q)$-weight of the \dkpc. in
Figure~\ref{fig:dk} is $(-1)^7 q^{2\cdot 8 + 1+3+4+3+3+3+2} (-y)^{3-2}$.

The proof of the lemma below is similar to that of Lemma~\ref{lem:wt}.
Here we define
\[\J =(1+yq, 1-q^2, 1+yq^3,1-q^4, \dots), \qquad
\J' =(1+y^{-1}q, 1-q^2, 1+y^{-1}q^3, 1-q^4, \dots).\]

\begin{lem}\label{lem:wt_yq}
For any nonnegative integer $k$, we have 
\[ \sum_{p\in \MD^*_k} \wt(p;\J-1,\J'-1) = 
q^{k(k+1)} \sum_{C\in\Dp{k}}\wt_{y,q^{-1}}(C). \]   
\end{lem}

By the same argument as in the proof of Proposition~\ref{thm:separate}
together with 
\[
\prod_{i\geq1} \! \frac{1}{1-q^{2i}}
\prod_{i\geq1} \! ({1-q^{2i}}) ({1+yq^{2i-1}}) 
\prod_{i\geq1} \! ({1-q^{2i}}) ({1+y^{-1}q^{2i-1}})
\! = \!  \prod_{i\geq1} \! (1-q^{2i})(1+yq^{2i-1})(1+y^{-1}q^{2i-1}),
\]
we obtain
the following.
\begin{prop}\label{thm:sepjtp}
For any nonnegative integer $k$, we have
\begin{equation}
  \label{eq:3}
\sum_{C\in\Dp{k}}\wt_{y,q}(C) \equiv 
\prod_{i\geq1} (1-q^{2i})(1+yq^{2i-1})(1+y^{-1}q^{2i-1})  \mod q^k.
\end{equation}
\end{prop}

Since the right-hand side of \eqref{eq:3} is one side of the triple product identity, it is natural to guess
\begin{equation}
  \label{eq:4}
 \sum_{C\in\Dp{k}}\wt_{y,q}(C) = \sum_{i=-k}^{k} y^i q^{i^2}.
\end{equation}

Using Lemma~\ref{lem:wt_yq} and the same argument as in \eqref{eq:6}, one can
see that \eqref{eq:4} is equivalent to Theorem~\ref{thm:main}. This is in fact
the way the authors first discovered Theorem~\ref{thm:main}. Notice also that by
Lemma~\ref{lem:wt_yq} and \eqref{eq:2}, we have
\[q^{k(k+1)} \sum_{C\in\Dp{k}}\wt_{y,q^{-1}}(C)=
\sum_{p\in \MD^*_k} \wt(p;\J-1,\J'-1)
=\sum_{p\in \SCH_k} \wt(p;\J,\J').\]
Thus \eqref{eq:4} is also equivalent to
\begin{equation}
  \label{eq:5}
\sum_{p\in \SCH_k} \wt(p;\J,\J') = \sum_{j=-k}^{k} y^j q^{k(k+1)-j^2}.
\end{equation}

We will prove this in the second part of this article.

\subsection{\texorpdfstring{Generalized $q$-secant numbers}{Generalized q-secant numbers}}

For two nonnegative integers $a$ and $b$, we define
\[ E_{n}^{a,b}(q) = [z^n] \left( \frac{1}{1} \cmo \frac{\qint{a+1}\qint{b+1}
    z}{1} \cmo \frac{\qint{a+2}\qint{b+2} z}{1} \cmo \cdots\right).\] 
Then $E_n^{0,0}(q)$ are the $q$-secant numbers as in \eqref{qsec},
and $E_n^{0,1}(q)$ are $q$-analogs of the tangent numbers, see \cite{josuat1}.
By \eqref{qsec}, we have
 \[ E_{n}^{0,0}(q) = \frac{1}{(1-q)^{2n}} 
 \sum_{k=0}^n \left( \tbinom{2n}{n-k} - \tbinom{2n}{n-k-1} \right)
q^{k(k+1)} T_k(q^{-1}),\] 
where $T_k(q) = \sum_{i=-k}^{k} (-q)^{i^2}$. Note that by \eqref{eq:gauss} we have
\[ \lim_{k\to\infty} T_k(q) = \frac{(q;q)_\infty}{(-q;q)_\infty},\] 
where we use
the usual notation $(a;q)_n = (1-a)(1-aq)\dotsm(1-aq^{n-1})$ and $(a;q)_\infty =
(1-a)(1-aq)\dotsm$.
We can generalize the above property to $E_{n}^{a,b}(q)$. Since the proofs of
the statements in the rest of this section are similar to those for the case
$a=b=0$, we will omit the proofs.
Let
\begin{align*}
\U_a &=(\qint{a+1}, \qint{a+2}, \ldots), & \U_b &=(\qint{b+1}, \qint{b+2}, \ldots),\\
\V_a &=(1-q^{a+1}, 1-q^{a+2}, \ldots), & \V_b &=(1-q^{b+1}, 1-q^{b+2}, \ldots).
\end{align*}
Then
\[
E_{n}^{a,b}(q) = \sum_{p\in\D_n} \wt(p;\U_a,\U_b) = \frac{1}{(1-q)^{2n}}
\sum_{k=0}^n \left( \tbinom{2n}{n-k} - \tbinom{2n}{n-k-1} \right)
\sum_{p\in \MD^*_k} \wt(p;\V_a-\one,\V_b-\one).
\]

For $(\lambda,A)\in \Dp k$, we define \[\wt_q^{a,b}(\lambda,A) = (-1)^{|A|} 
q^{2|\lambda| + \norm{A} + a\cdot h(A) + b \cdot v(A)},\] where $h(A)$
(resp.~$v(A)$) is the number of horizontal (resp.~vertical) arrows. Then 
we have
\[ \sum_{p\in \MD^*_k} \wt(p;\V_a-1,\V_b-1) = 
q^{k(k+1+a+b)} \sum_{(\lambda,A)\in\Dp{k}}\wt_{q^{-1}}^{a,b}(\lambda,A). \]   

\begin{prop}
For any nonnegative integer $k$, we have
\[ \sum_{C\in\Dp{k}}\wt_q^{a,b}(C) \equiv \prod_{i\geq1} \frac{1}{1-q^{2i}}
\prod_{i\geq a} ({1-q^{i}})\prod_{i\geq b} ({1-q^{i}}) 
=\frac{1}{(q;q)_a (q;q)_b} \cdot \frac{(q;q)_\infty}{(-q;q)_\infty}  \mod q^k.\]
\end{prop}

\begin{thm}
  For nonnegative integers $a$ and $b$, there is a unique family
  $\{T_k^{a,b}(q)\}_{k\geq0}$ of functions of $q$ such that
  \[ E_{n}^{a,b}(q) = \frac{1}{(1-q)^{2n}} \sum_{k=0}^n \left( \tbinom{2n}{n-k}
    - \tbinom{2n}{n-k-1} \right) q^{k(k+1+a+b)} T_k^{a,b}(q^{-1}).\] Moreover,
  for all $k\geq0$, we have that $T_k^{a,b}(q)$ is a polynomial in $q$, and
\[
T_k^{a,b}(q) \equiv
\frac{1}{(q;q)_a (q;q)_b} \cdot \frac{(q;q)_\infty}{(-q;q)_\infty} \mod q^k,
\]
which implies
\[ \lim_{k\to\infty} T_k^{a,b}(q) = 
\frac{1}{(q;q)_a (q;q)_b} \cdot \frac{(q;q)_\infty}{(-q;q)_\infty}.\]
\end{thm}

The uniqueness of $\{T_k^{a,b}(q)\}_{k\geq0}$ comes from the fact that
the matrix $\left( \tbinom{2n}{n-k}  - \tbinom{2n}{n-k-1} \right)_{n,k\geq0}$ is
invertible. 


\part{\texorpdfstring{New Touchard-Riordan--like formulas via the general case
    of the triple product identity}{New Touchard-Riordan–like formulas via the
    general case of the triple product identity}}

\section{Preliminaries}

We have seen in Lemma~\ref{lem2} from the previous part that the relation between the coefficients of
$S_\lambda(z)$ and those of $T_\lambda(z)$ can be proved combinatorially. Let us show that the
result can be proved analytically as well.

\begin{lem}[A reformulation of Lemma \ref{lem2}] \label{lem2bis}
Given a sequence $\lambda=\{\lambda_n\}_{n\geq1}$, we define
$\mu=\{\mu_n\}_{n\geq0}$ and $\nu=\{\nu_n\}_{n\geq0}$ such that:
\[
\sum_{n=0}^\infty \mu_nz^n = S_\lambda(z), \qquad
\sum_{n=0}^\infty \nu_nz^n = T_\lambda(z).
\]
Then for any $n\geq0$ we have the relation
$\mu_n = \sum\limits_{k=0}^n \left( \binom{2n}{n-k} - \binom{2n}{n-k-1} \right) \nu_k$.
\end{lem}

\begin{proof}
In each fraction of $T_\lambda(z)$, divide both the numerator and denominator by $1+z$. This gives
\begin{align*}
  T_\lambda(z) &=  \frac{(1+z)^{-1}}{1} \cmo \frac{\lambda_1 z(1+z)^{-1} }{1+z} \cmo \frac{\lambda_2 z}{1+z}
                  \cmo \frac{\lambda_3 z}{1+z} \cmo \dotsb \\[1mm]
   &=  \frac{(1+z)^{-1}}{1} \cmo \frac{\lambda_1 z(1+z)^{-2} }{1} \cmo \frac{\lambda_2 z(1+z)^{-1} }{1+z}
                  \cmo \frac{\lambda_3 z}{1+z} \cmo \dotsb \\[1mm]
   &=  \frac{(1+z)^{-1}}{1} \cmo \frac{\lambda_1 z(1+z)^{-2} }{1} \cmo \frac{\lambda_2 z(1+z)^{-2} }{1}
                  \cmo \frac{\lambda_3 z(1+z)^{-1} }{1+z} \cmo \dotsb,
\end{align*}
{\it i.e.} this shows an equivalence of continued fractions so that
\[
   T_\lambda(z) = \frac{1}{1+z}  S_\lambda\left(\frac{z}{(1+z)^2}\right).
\]
Let $u=z(1+z)^{-2}$, and note that $u=z+O(z^2)$. Let us write
\begin{equation} \label{rel_lem1}
S_\lambda(u)=(1+z)T_\lambda(z).
\end{equation}
Now it suffices to show that
\begin{equation} \label{gen_del}
(1+z)z^k = \sum_{n=k}^\infty
\left( \binom{2n}{n-k} - \binom{2n}{n-k-1} \right) u^n,
\end{equation}
because knowing this makes possible to identify coefficients in both sides
of \eqref{rel_lem1} with respect to the new variable $u$, thus completing the
proof.

Showing \eqref{gen_del} can be done using one of the various forms of Lagrange
inversion. For example, from \cite[Section 3.8, Theorem A]{comtet}, we directly
obtain that the coefficient of $u^n$ in $z^k$ is
\[
  [u^n]z^k = \frac kn [t^{n-k}] (1+t)^{2n} = \frac kn \binom{2n}{n-k},
\]
and eventually the coefficient of $u^n$ in $(z+1)z^k$ is
\begin{equation*}\begin{split}
  [u^n](z+1)z^k &= \frac {k+1}n \binom{2n}{n-k-1} + \frac {k}n \binom{2n}{n-k} \\[1mm]
   &=  \frac{2k+1}{n+k+1}\binom{2n}{n-k} = \binom{2n}{n-k} - \binom{2n}{n-k-1}.
\end{split}\end{equation*}
This completes the proof.
\end{proof}

Note that the intermediate form $\frac{2k+1}{n+k+1}\binom{2n}{n-k}$ in the above demonstration might
seem more compact than the difference of binomial coefficients. The latter form has other
benefits, for example it makes apparent that these numbers satisfy a simple recurrence
similar to binomial coefficients.

We will need a technical result of uniform convergence for continued fractions. This will be used in
Sections \ref{secjtp} and \ref{secgeno}.

\begin{lem} \label{lemconv}
Let $\{w_n\}_{n\geq1}$ be an arbitrary sequence of formal power series in the variable
$z$, and let $t_n(z)$ be the $n$th modified approximant of $T_\lambda(z)$ with respect
to $zw_n$, {\it i.e.} the finite continued fraction:
\[
t_n(z) =   \frac{1}{1+z} \cmo \frac{\lambda_1 z}{1+z} \cmo \frac{\lambda_2 z}{1+z}
               \cmo \dotsb \cmo  \frac{\lambda_n z}{1+z w_n}.
\]
Then $t_n(z)$ converges to $T_\lambda(z)$ with respect to the usual ultrametric distance
on formal power series.
\end{lem}

\begin{proof}
Let us prove that the first $n$ Taylor coefficients of $t_n(z)$ do not depend on $w_n$.
To begin, we have:
\[
 \frac{\lambda_nz}{1+zw_n} = \lambda_nz + O(z^2).
\]
Going one step further, we have
\[
  \frac{\lambda_{n-1}z}{1+z}  \cmo  \frac{\lambda_nz}{1+zw_n} = \frac{\lambda_{n-1}z}{1+z-\lambda_nz+O(z^2)}
  = \frac{\lambda_{n-1}z}{1+z-\lambda_nz} + O(z^3).
\]
After $n$ such steps, we can obtain an expression of $t_n(z)$ that does not depend on $w_n$ up to a $O(z^{n+1})$.
A particular choice of $w_n$ is such that $t_n(z)=T_\lambda(z)$, namely $w_n=\lambda_{n+1}T_{\lambda'}(z)$ where
$\lambda'$ is the shifted sequence $\{\lambda_{n+2+k}\}_{k\geq0}$.
It follows that for an arbitrary choice of $w_n$, we have
\[
 t_n(z) - T_\lambda(z) = O(z^{n+1}).
\]
This proves the convergence.
\end{proof}

Of course, it would be possible to give an analogous result for S-fractions, but we will not
need it. Note that this statement is no longer true when $w_n$ are formal Laurent series,
and in some proofs below, a key point is to check that some $w_n$ which are {\it a priori}
Laurent series, are actually well-defined at 0.

\section{\texorpdfstring{A ``finite version'' of Jacobi's triple product}{A “finite version” of Jacobi's triple product}}
\label{secjtp}

By Lemma~\ref{lem2bis}, Theorem~\ref{thm:main} is equivalent to the following
theorem.  This is also equivalent to \eqref{eq:4} and \eqref{eq:5}, and as we
have seen, these identities give Jacobi's triple product identity when $k$ tends to infinity.

\begin{thm} \label{jtp_th}
There holds
\begin{equation} \label{jtp_frac}
     \sum_{k=0}^\infty z^k \sum_{j=-k}^{k} y^j q^{k(k+1)-j^2}
   =
      \frac{1}{1+z} \cmo \frac{(1+qy)(1+q\y)z}{1+z} \cmo
      \frac{(1-q^2)^2z}{1+z} \cmo \dotsb , 
\end{equation}
the continued fraction being $T_\lambda(z)$ with
$\lambda_n=(1+q^ny)(1+q^ny^{-1}) $ for odd $n$ and $\lambda_n=(1-q^n)^2$ for
even $n$.
\end{thm}

\begin{proof}
We show that both sides
satisfy a common functional equation. Indeed, if $H(z)$ is the left-hand side
of \eqref{jtp_frac} and $c_{k,j}(z)=z^ky^jq^{k(k+1)-j^2}$, we have
\[
 H(z)=\sum_{k,j\in\mathbb{Z}, \; k\geq |j|} c_{k,j}(z), \qquad c_{k+1,j}(z)=zq^2c_{k,j}(zq^2),
\]
and it follows that
\[
  H(z) - zq^2H(zq^2) = \sum_{j\in\mathbb{Z}} c_{|j|,j}(z) = \frac{1}{1-yqz} + \frac{1}{1-y^{-1}qz} - 1.
\]
To show the latter equality, note that when $k=|j|$ the term $k^2$ cancels with $-j^2$ in $c_{k,j}(z)$
and splitting the $j$-sum according to the sign of $j$ gives two geometric series. So $H(z)$
is the unique formal power series in $z$ satisfying the functional equation
\begin{equation} \label{funeqH}
  H(z) = \frac{1}{1-yqz} + \frac{1}{1-y^{-1}qz} - 1 + zq^2H(zq^2).
\end{equation}
The uniqueness comes from the fact that \eqref{funeqH} gives a relation between $a_n$ and $a_{n-1}$ if $a_n$
is the coefficient of $z^n$ in $H(z)$, so all coefficients are determined by $a_{0}=1$.
It remains only to show that the continued fraction in the right-hand side
of \eqref{jtp_frac} satisfies the same functional equation, which is done in
a separate lemma below.
\end{proof}

\begin{lem} \label{funeq1}
Let $\lambda$ be the sequence in Theorem \ref{jtp_th}, then we have
\begin{equation} \label{funeqT}
  T_\lambda(z) = \frac{1}{1-yqz} + \frac{1}{1-y^{-1}qz} - 1 + zq^2 T_\lambda(zq^2).
\end{equation}
\end{lem}

\begin{proof}
We will identify $2\times2$ matrices and M\"obius transformations in the usual way:
\[
\begin{pmatrix} a&b \\ c&d \end{pmatrix} [X] = \frac{aX+b}{cX+d},
\]
{\it i.e.} we use a bracket notation for the evaluation of a Möbius transformation.
A continued fraction can be obtained by iterating such transformations.
In the present case, we have
\[
  \tfrac{1}{1+z- \tfrac{ (1+wqy)(1+wqy^{-1} )z  }{ 1+z - (1-wq^2)^2zX }   }
=
  \tfrac{ z(1-wq^2)^2X - (1+z) }{ z(1+z)(1-wq^2)^2 X + (1+wqy)(1+wqy^{-1})z-(1+z)^2  },
\]
so we can introduce the matrix
\begin{equation} \label{def_m}
M(w,z) =
  \left(\begin{mmatrix}
    z(1-wq^2)^2      &   -1-z \\
    z(1+z)(1-wq^2)^2 &  (1+wqy)(1+wqy^{-1})z-(1+z)^2
  \end{mmatrix}\right),
\end{equation}
and we have
\begin{equation} \label{prod0}
 T_\lambda(z) = M(1,z) M(q^2,z) M(q^4,z)\dotsb .
\end{equation}
The partial products are just the convergents of the continued fraction.  More
precisely, the infinite product is convergent in the following sense: the
partial products are M\"obius transformations, and these converge pointwise to
the formal power series in the left-hand side of \eqref{prod0}.

Let $S$ be the matrix
\[
  S = \begin{pmatrix} zq^2 & \tfrac{1}{1-zqy}+\tfrac{1}{1-zqy^{-1}}-1 \\ 0 & 1    
\end{pmatrix},
\]
so that the functional equation to be proved is $T_\lambda(z) = S\big[ T_\lambda(zq^2) \big] $. This can be written:
\begin{equation}  \label{funeqprod}
 M(1,z) M(q^2,z) M(q^4,z) \dotsm = S M(1,zq^2) M(q^2,zq^2) M(q^4,zq^2) \dotsm .
\end{equation}
By examining the previous equation, it is natural to introduce a matrix $\Omega_n$ by
\begin{equation} \label{def_omega}
\Omega_n =  M(q^{2n-2},z)^{-1} \dotsm M(1,z)^{-1}  S M(1,zq^2) \dotsm M(q^{2n-2},zq^2),
\end{equation}
where we understand that only even powers of $q$ appear within the dots. In particular,
we have $\Omega_0=S$. It can be calculated explicitly, as given in Lemma~\ref{lem_omega}
below, so that we obtain:
\[
 \Omega_n [0] = \frac{1-z^2q^2}{zq^{2n+1}(2qz - y - y^{-1}) +1-z^2q^2}.
\]
The important point is that from this closed form, we can check that $\Omega_n [0]$ is well-defined at $z=0$
({\it i.e.} it has no pole at $z=0$). Let $w_n=\Omega_n [0] $, by definition of $\Omega_n$ we have:
\begin{equation} \label{matrixeq1}
  M(1,z) \dotsm M(q^{2n-2},z) [w_n] = S M(1,zq^2) \dotsm M(q^{2n-2},zq^2) [0] ,
\end{equation}
and at this point it remains only to let $n$ tend to infinity in \eqref{matrixeq1} to
prove \eqref{funeqprod}, which is a rewriting of \eqref{funeqT}. The only subtlety is
in the left-hand side, where we need the fact that $w_n$ is indeed a formal power series
in $z$ (as opposed to a formal Laurent series) to take the limit. Indeed, in this
case we can apply Lemma~\ref{lemconv}, so that the left-hand side of \eqref{matrixeq1}
converges to $T_\lambda(z)$.
\end{proof}

\begin{lem}  \label{lem_omega}
The matrix $\Omega_n$ defined in \eqref{def_omega} has the explicit form:
\begin{equation} \label{omegan}
 \Omega_n =
\frac{
\left(\begin{mmatrix}
   q^2z ( 2q^{2n}  - zq^{2n+1}(y+y^{-1})  + z^2q^2 -1  )    & 1-z^2q^2 \\
  (1-q^{2n})^2(z^2q^2-1)zq^2 & zq^{2n+1}(2qz - y - y^{-1}) +1-z^2q^2
\end{mmatrix}\right)
}{(1-yzq)(1-y^{-1}zq)}.
\end{equation}
\end{lem}
\begin{proof}
  Although calculations are quite cumbersome, there is a straightforward
  recursive verification of the given expression, using the relations $\Omega_{0}=S$ and
\begin{equation} \label{omega_rec}
 \Omega_{n+1} =  M(q^{2n},z)^{-1}  \Omega_{n}  M(q^{2n},zq^2)
\end{equation}
for $n\geq 0$.
There are $4$ coefficients in $\Omega_{n}$, each appears as a sum of $4$ terms when we expand the previous equation,
and each of these term is a product of $3$ coefficients of the matrices in \eqref{def_m} and \eqref{omegan}. So there is a
small ``explosion'' of the size of computations to perform. However, this is a verification that can be done with no particular
cleverness, since expanding everything in \eqref{omega_rec} will clearly makes possible a term-by-term identification of
both sides. We omit details and invite the unconvinced reader to use some computer algebra system for checking that
the lemma is true.
\end{proof}

\section{\texorpdfstring{New Touchard-Riordan--like formulas}{New Touchard-Riordan–like formulas}}

From Theorem \ref{jtp_th}, we can derive a whole family of S-fractions having
a Touchard-Riordan--like formula, as given in the theorem below. A very
interesting property of these is that there are exponential generating functions
linked with trigonometric functions. The theorem below is also a wide
generalization of the result in \eqref{qsec}, which is related with a $q$-analog
of secant numbers having exponential generating function $\sec(z)$.

Note that in the definition of $\qint{n}=(1-q^n)/(1-q)$, $n$ can be any number, not
necessarily an integer.

\begin{thm} \label{theoS}
For any numbers $a$ and $b$, we define $\mu_n(a,b,q)$ by
\begin{equation} \label{def_mu}
 \sum_{n=0}^\infty \mu_n(a,b,q) z^n = S_\lambda(z),
\quad \hbox{ where } \quad
\lambda_n =
\begin{cases}
  \left[ nb+a\right]_q\left[nb-a\right]_q
     & \hbox{ if $n$ is odd},\\[2mm]
  \left[nb \right]_q^2
     & \hbox{ if $n$ is even}.\\
\end{cases}
\end{equation}
Then we have
\[
\mu_n(a,b,q) = \frac 1{(1-q)^{2n}} \sum_{k=0}^n
   \left(\tbinom{2n}{n-k} - \tbinom{2n}{n-k-1}\right)
   \sum_{j=-k}^{k} (-1)^j q^{  aj+ b(k(k+1)-j^2) },
\]
and
\begin{equation} \label{coscos}
\sum_{n=0}^\infty \mu_n(a,b,1) \frac{z^n}{n!} = \frac{\cos(az)}{\cos(bz)}.
\end{equation}
Moreover, $\mu(a,b,q)$ is a polynomial in $q$ with nonnegative coefficients in the following situations:
\begin{itemize}
 \item $a$ and $b$ are integers such that $0\leq a <b$,
 \item $a$ and $b$ are half-integers such that $0\leq a <b$ (the set of half-integers is $\frac12+\mathbb{Z}$).
\end{itemize}
\end{thm}

\begin{remark}
Before proving this, note that if $nb\pm a=0$ for some odd $n\geq1$, so that $\lambda_n=0$, then on one
hand $S_\lambda(z)$ is a finite continued fraction, and on the other hand the exponential generating
function in \eqref{coscos} is a polynomial in $\cos(bz)$.
\end{remark}

\begin{proof}
  Consider the identity obtained by substituting $(y,q)$ with
  $(-q^{a},q^{b})$ in \eqref{jtp_frac}, and apply Lemma~\ref{lem2bis} to this
  T-fraction. This gives the desired formula for $(1-q)^{2n}\mu_n(a,b,q)$.

  One easily check that $\lambda_n$ is a $q$-integer when
  $a$ and $b$ are both integers or both half-integers such that $0 \leq a < b$.
  It follows that $\mu_n(a,b,q)$ is indeed a polynomial in $q$ with nonnegative coefficients
  in this situation.

  It remains only to obtain the exponential generating function of
  $\mu_n(a,b,1)$ as a ratio of cosines. Actually, this was essentially known by
  Stieltjes \cite{stieltjes} via analytical methods.  It is also possible to
  prove this going through an addition formula satisfied by $\cos(az)/\cos(bz)$,
  and using a theorem of Stieltjes and Rogers, see for example in \cite[Chapter
  5]{goulden} (this method is generally well-suited for trigonometric
  functions).
\end{proof}

Note that this makes apparent that a function such as $\cos(z/2)/\cos(3z/2)$ is the
exponential generating function for a sequence of nonnegative integers.  
Theorem~\ref{theoS} opens some combinatorial problems, for a better understanding of these
quantities $\mu_n(a,b,q)$. We can ask if there is a model 
from which both the ordinary generating function and the exponential one (for
$q=1$) can be obtained. It would be quite remarkable to obtain the
continued fraction on one side and the trigonometric function on the other side.

We can give an alternative proof of \eqref{coscos} using generating functions. This will take the rest of this
section, and it is an adaptation of a result from \cite{dumont}. Indeed, Dumont proves in Corollary 3.3 of his article \cite{dumont}
a continued fraction related with the exponential generating function $\cosh(z)/\cosh(2z)$, which is essentially
the case $a=1,b=2$ of our general result in \eqref{coscos}. We show here that his argument can be extended, and
it is a very interesting proof: first for its originality, second because it is elementary in the sense that it does not rely
on any important theorem such as the one of Stieltjes and Rogers.

Here, for a sequence $\{u_n\}_{n\geq0}$ the series $\sum_{n=0}^\infty u_n\frac{z^n}{n!}$ will be called
its exponential generating function and the series $\sum_{n=0}^\infty u_nz^{n+1}$ its ordinary generating function.
We will denote by $\mathcal{L}[f]$ the formal Laplace transform of a series $f$, which send
an exponential generating function to the ordinary one, {\it i.e.}
\[
   f(z)=\sum_{n=0}^\infty u_n \frac{z^n}{n!}  \quad \Longleftrightarrow \quad \mathcal{L}[f](z) = \sum_{n=0}^\infty u_n z^{n+1}.
\]
Note that the general term in the latter series is $u_nz^{n+1}$ and not $u_nz^n$. This is convenient
in this context, because of the following:

\begin{lem} \label{dumlem1}
Let $f(z)=\sum_{n=0}^\infty u_n \frac{z^n}{n!} $ be a series, then we have:
\[
  \mathcal{L}\left[ f(z) e^{az} \right](z) = \mathcal{L}[f]\Big( \frac{z}{1-az}  \Big).
\]
\end{lem}
\begin{proof}
 Let us form the new sequence $\{v_n\}_{n\geq0}$ where $v_n=\sum_{k=0}^n \binom n k u_k a^{n-k}$. Its exponential
generating function is $f(z) e^{az}$. Using the expansion
\[
 \Big(\frac{z}{1-az}\Big)^{k+1} = \sum_{n\geq k} \binom n k a^{n-k}z^{n+1},
\]
we obtain its ordinary generating function as $\mathcal{L}[f]\big( \tfrac{z}{1-az}  \big)$.
\end{proof}

\begin{lem}
There is a unique series $f(z)$ such that $f(z)e^{(a+1)z}$ is even and $(1-f(z))e^{az}$ is odd. It is explicitly
given by
\begin{equation} \label{def_f}
  f(z)= \frac{\cosh(az)}{\cosh(z)} e^{-(a+1)z} .
\end{equation}
\end{lem}
\begin{proof} With $f(z)$ defined in \eqref{def_f}, clearly $f(z)e^{(a+1)z}$ is even, and we have
\[
 (1-f(z))e^{az} = \frac{e^{(a+1)z} - e^{-(a+1)z} }{ 2\cosh(z) },
\]
so the proposed function satisfies the conditions in the lemma. To prove the uniqueness,
suppose that $g(z)=\sum_{n=0}^\infty u_n \frac{z^n}{n!}$ satisfies the same conditions, we show that the coefficients $u_n$
can be computed recursively, hence are fully determined. To begin, $u_0=0$ since $(1-g(z))e^{az}$ is odd. More
generally:
\begin{itemize}
 \item cancel the coefficient of $z^{2n+1}$ in $g(z)e^{(a+1)z}$ gives an expression of $u_{2n+1}$ which only depends
         on $a$ and $u_0,\dots,u_{2n}$,
\item  cancel the coefficient of $z^{2n}$ in $(1-g(z))e^{az}$ gives an expression of $u_{2n}$ which only depends on
         $a$ and $u_0, \dots , u_{2n-1}$.
\end{itemize}
This implies the uniqueness.
\end{proof}

Out of the context, it might seem curious to characterize a series by such a
criterion as in the previous lemma.  In Dumont's article \cite{dumont}, the
characterization of $\cosh(z)/\cosh(2z)$ in these terms
appears naturally since it is an interpretation of the Arnold-Seidel triangles
for Springer numbers at the level of exponential generating functions. The
important property of this characterization of $f(z)$ is that it is easily
translated in terms of $\mathcal{L}[f](z)$, using Lemma~\ref{dumlem1} on one
side, and on the other side, the fact that $f$ is odd or even if and only if $\mathcal{L}[f](z)$
is respectively even or odd.

To use this characterization of $\mathcal{L}[f](z)$ we will need the following:

\begin{lem}[First and second contractions of continued fractions]  \label{contractions}
  The following three representations of a series $F(z)$ are equivalent:
\begin{align}
  F(z) &= \frac z 1 \cpo \frac{c_1z}{1} \cpo \frac{c_2z}{1} \cpo \frac{c_3z}{1} \cpo \dotsb, \\
  F(z) &= \frac z {1+c_1z} \cmo \frac{c_1c_2z^2}{1+(c_2+c_3)z} \cmo \frac{c_3c_4z^2}{1+(c_4+c_5)z} \cmo \dotsb, \label{cont1}  \\
  F(z) &= z - \frac {c_1z^2} {1+(c_1+c_2)z} \cmo \frac{c_2c_3z^2}{1+(c_3+c_4)z} \cmo \frac{c_4c_5z^2}{1+(c_5+c_5)z} \cmo \dotsb . \label{cont2}
\end{align}
\end{lem}

This is common knowledge in continued fractions, see for example \cite[Chapter 1]{wall}.
The following statement and its proof are adapted from Dumont's article~\cite{dumont}.

\begin{lem} \label{dumlem4}
 Let $c_0=0$ and $F(z)$ be the series in the previous lemma. The following conditions are equivalent:
\begin{itemize}
 \item $F\big(\frac{z}{1-(a+1)z}\big)$ is odd and $\frac{z}{1-az} - F\big(\frac{z}{1-az}\big)$ is even,
 \item for all $n\geq0$, there holds $c_{2n}+c_{2n+1}=a+1$, $c_{4n+1}+c_{4n+2}=2a$, $c_{4n+3}+c_{4n+4}=0$.
 \item for all $n\geq0$, there holds $c_{4n}=-2n$, $c_{4n+1}=2n+1+a$, $c_{4n+2}=-2n-1+a$, $c_{4n+3}=2n+2$.
\end{itemize}
\end{lem}

\begin{proof}
Substitute $z$ with $\frac{z}{1-(a+1)z} $ in \eqref{cont1}, this gives:
\begin{equation}  \label{Fexp} \begin{split}
 F\Big(\frac{z}{1-(a+1)z}\Big) = \frac z {1-(a+1)z+c_1z}  \cmo  &  \frac{c_1c_2z^2}{1-  (a+1)z+(c_2+c_3)z}  \cmo \\
   & \qquad  \frac{c_3c_4z^2}{1-(a+1)z+(c_4+c_5)z} \cmo \dotsb . 
\end{split}\end{equation}
This continued fraction is odd if and only if the terms $-(a+1)z+c_1z$, $-(a+1)z+(c_2+c_3)z$, {\it etc.} are 0, more
precisely $c_1=a+1$ and $c_{2n}+c_{2n+1}=a+1$ for all $n\geq0$. Then, substitute
$z$ with $\frac{z}{1-az}$ in \eqref{cont2}, this gives:
\begin{equation*}\begin{split}
   \frac{z}{1-az} - F\left(\frac{z}{1-az} \right) & = \frac {c_1z^2} {(1-az)^2+(c_1+c_2)z(1-az)} \cmo \frac{c_2c_3z^2}{1+(c_3+c_4)z} \cmo \\
    & \qquad\qquad \frac{c_4c_5z^2}{(1-az)^2+(c_5+c_6)z(1-az)} \cmo \frac{c_6c_7z^2}{1+(c_7+c_8)z} \cmo\dotsb .
 \end{split}
\end{equation*}
After expanding the terms $(1-az)^2$ and $z(1-az)$, we can see that this continued fraction is even if and only if 
$c_1+c_2=2a$, $c_3+c_4=0$, $c_5+c_6=2a$, $c_7+c_8=0$, {\it etc.}, more precisely
$c_{4n+3}+c_{4n+4}=0$, $c_{4n+1}+c_{4n+2}=2a$ for all $n\geq0$.

Thus we have proved the equivalence between the first two items in the list. The equivalence between the last two
items is elementary and details are omitted.
\end{proof}

With all the previous lemmas, we can prove the second part of Theorem~\ref{theoS}, {\it i.e.}
obtain the exponential generating function \eqref{coscos} from \eqref{def_mu}.

\begin{proof}
  After the substitution $(a,b)\mapsto(ai,i)$, 
  what we want can be rephrased as:
\begin{equation}  \label{rephrased}
 \mathcal{L}\left[ \frac{\cosh(az)}{\cosh(z)} \right](z) =
 \frac z 1 \cpo \frac{(1-a^2)z^2}{1} \cpo \frac{4z^2}{1} \cpo \frac{(9-a^2)z^2}{1} \cpo \frac{16z^2}{1} \cpo \dotsb .
\end{equation}
There is no loss of generality in this substitution since we have
\[
 \mu_n(a,b,1) = b^{2n} \mu_n \big( \tfrac ab,1,1 \big).
\]
Let $f(z)$ be the series in \eqref{def_f}, and $F=\mathcal{L}[f]$. It follows that $F(z)$ satisfies the conditions in
Lemma~\ref{dumlem4}. We have:
\begin{equation*} \begin{split}
 \mathcal{L}\left[ \tfrac{\cosh(az)}{\cosh(z)} \right](z) &= 
\mathcal{L}\left[ f(z)e^{(a+1)z} \right](z) = F\big( \tfrac{z}{1-(a+1)z} \big)
   = \frac z1 \cmo \frac{c_1c_2z^2}{1} \cmo \frac{c_3c_4z^2}{1} \cmo \dotsb , \\
\end{split}
\end{equation*}
where the last equality comes from \eqref{Fexp} and the relations satisfied by
$c_i$ in the second point in the list of Lemma~\ref{dumlem4}. Then from the
explicit form in the third point of the same list, we have $c_1c_2=a^2-1$,
$c_3c_4=-4$, {\it etc.} and this way we obtain
\eqref{rephrased}, thus completing the proof.
\end{proof}

\section{\texorpdfstring{A Touchard-Riordan--like formula for $q$-Genocchi numbers}{A Touchard-Riordan--like formula for q-Genocchi numbers}}
\label{secgeno}

In this section, we prove a Touchard-Riordan--like formula for a $q$-analog of Genocchi numbers.
Whereas the previous results were related with the triple product identity, the one in this section
will be related with the following, also due to Jacobi:
\begin{equation} \label{jacobicube}
   \prod_{n=1}^\infty (1-q^n)^3 = \sum_{i=0}^{\infty} (-1)^i (2i+1) q^{\binom{i+1}2 }.
\end{equation}

Genocchi numbers $\{ G_{2n} \}_{n\geq1}$ can be defined through their generating function:
\begin{equation*} 
    \sum_{n=1}^\infty G_{2n} \frac{z^{2n}}{ (2n)! } = z \cdot \tan\Big(\frac z2 \Big).
\end{equation*}
The $q$-analog $G_{2n}(q)$ of these numbers can be defined by:
\begin{equation} \label{gen_frac}
    \sum_{n=1}^{\infty}  G_{2n}(q) z^{n-1} =
      \frac 11  \cmo \frac{[1]_{q}[1]_{q}z}1 \cmo
      \frac{[1]_{q}[2]_{q}z} 1 \cmo
      \frac{[2]_{q}[2]_{q}z} 1 \cmo
      \frac{[2]_{q}[3]_{q}z} 1 \cmo \dotsb ,
    \end{equation}
    {\it i.e.} the generating function is $S_\lambda(z)$ where
    $\lambda_n = [  (n+1)/2 ]_q ^2  $ for odd $n$
    and $\lambda_n = [ n/2]_q [ n/2+1 ]_q $ for even $n$.
    The fact that $G_{2n}(1)=G_{2n}$ follows from a continued fraction 
    given by Viennot \cite{viennot}, which is the case $q=1$ of $S_\lambda(z)$ as above.
    This $G_{2n}(q)$ was also considered in \cite{BJS}, where it is defined combinatorially
    in terms of Dumont permutations.

The following result was conjectured in the first author's Phd thesis \cite{josuat2}. We prove it in this
section, using the same method as in the previous section, and make the link with \eqref{jacobicube} in the next section.

\begin{thm} \label{qgeno}
Let
\begin{equation} \label{def_Pk}
  P_k=\sum_{i=0}^k (-1)^i(2i+1) q^{ \binom{k+1}{2} - \binom{i+1}{2} },
 \qquad
  Y_k = \frac{ P_{k-1} + 2P_{k} + P_{k+1} } {q-1}.
\end{equation}
Then $Y_k$ is a polynomial in $q$ such that for any $n\geq 0$ we have
\begin{equation}\label{eq:GY}
   G_{2n+2}(q) = \frac 1{(1-q)^{2n}} \sum_{k=0}^n
   \left(\binom{2n}{n-k} - \binom{2n}{n-k-1}\right) Y_k.
\end{equation}
\end{thm}

\begin{remark} It is possible to rewrite this formula directly in terms of
  $P_k$. Denoting $C_n$ the $n$th Catalan number
  $\frac{1}{n+1}\binom{2n}{n}$, this is such that for any $n\geq1$ we have
\begin{equation} \label{qgenoPk}
  G_{2n}(q) = \frac{1}{(q-1)^{2n-1}} \left(  C_{n-1} + \sum_{k=0}^{n}
   \left(\binom{2n}{n-k} - \binom{2n}{n-k-1}\right) P_k    \right).
\end{equation}
\end{remark}

\bigskip

\begin{proof}
Let $\lambda'=\{\lambda'_n\}_{n\geq0}$ where $\lambda'_n=(1-q)^2\lambda_n$.
If we substitute $z$ with $(1-q)^2z$ in \eqref{gen_frac}, the resulting S-fraction in the right-hand side
is $S_{\lambda'}(z)$. Apply Lemma~\ref{lem2bis} to this S-fraction, it follows that
\begin{equation*}
   (1-q)^{2n} G_{2n+2}(q) =  \sum_{k=0}^n
   \left(\binom{2n}{n-k} - \binom{2n}{n-k-1}\right) W_k
\end{equation*}
where the sequence $\{W_k\}_{k\geq0}$ is defined by
\begin{equation} \label{W_frac}
   \sum_{k=0}^{\infty} W_k z^{k} = T_{\lambda'}(z).
\end{equation}
It remains only to show that $Y_k=W_k$, and it will be done by showing that the
generating functions satisfy a common functional equation. To this end, we can first
obtain the generating function of the sequence $\{P_k\}_{k\geq0}$:
\begin{equation*}\begin{split}
 \sum_{k=0}^\infty P_kz^k &= \sum_{0\leq i\leq k} (-1)^i(2i+1)q^{\binom{k+1}2-\binom{i+1}2}z^k \\
 &= \sum_{0\leq i,j} (-1)^i(2i+1)q^{ \frac{j(j+2i+1)}2 }z^{i+j} = \sum_{j=0}^{\infty} \frac{1-zq^j}{(1+zq^j)^2}q^{\binom{j+1}2}z^j,
\end{split}\end{equation*}
where we have introduced the index $j=k-i$ and used the expansion
\begin{equation*}
  \sum_{i=0}^{\infty} (2i+1) x^i = \frac{1+x}{(1-x)^2}.
\end{equation*}
To derive the generating function of $\{Y_k\}_{k\geq0}$, note that the natural extension when $k<0$ is such that
$Y_k=0$ if $k<-1$, and $Y_{-1}=P_0/(q-1)$ where $P_0=1$. It is practical to take into account this $Y_{-1}$,
and we have:
\begin{equation*}
  (q-1) \sum_{k=0}^{\infty} Y_{k-1} z^k = \sum_{k=0}^{\infty} (P_{k-2} + 2P_{k-1} + P_k)z^k
  = (1+z)^2 \sum_{k=0}^{\infty} P_k z^k.
\end{equation*}
We thus obtain:
\begin{equation} \label{Q_gen}
1 + z(q-1)\sum\limits_{k=0}^\infty Y_kz^k
= (1+z)^2 \sum_{k=0}^\infty \frac{1-zq^k}{(1+zq^k)^2} q^{\binom {k+1}2} z^k.
\end{equation}
As a consequence, if we define $F(z)$ as
\begin{equation}  \label{deffz}
F(z) = \frac{1}{(1+z)^2} \left( 1 + z(q-1)\sum\limits_{k=0}^\infty Y_kz^k \right)
= \sum_{k=0}^\infty \frac{(1-zq^k)}{(1+zq^k)^2} q^{\binom {k+1}2} z^k,
\end{equation}
then it is the unique formal power series in $z$ such that
\begin{equation*}
F(z) = \frac{1-z}{(1+z)^2} + qzF(qz).
\end{equation*}
As in the case of $H(z)$ previously seen, the uniqueness comes from the fact that the previous equation
gives a relation between $a_n$ and $a_{n-1}$ if $a_n$ is the coefficient of $z^n$ in $F(z)$,
hence these coefficients are determined by $a_0=1$.

At this point, it remains only to show that the series
\begin{equation*}
   G(z)  =  \frac{1}{(1+z)^2} \big( 1 + z(q-1) T_{\lambda'}(z) \big)  = \frac{1}{(1+z)^2} \left( 1 + z(q-1)\sum\limits_{k=0}^\infty W_kz^k \right)
\end{equation*}
satisfies the same functional equation as $F(z)$. This is done in a separate lemma below,
and completes the proof since it follows that $F(z)=G(z)$ and $Y_k=W_k$.
\end{proof}

\begin{lem} \label{funeqGlem}
Let us define
\begin{equation*}
G(z) = \frac{1}{(1+z)^2} \big( 1 + z(q-1) T_{\lambda'}(z) \big)    
\end{equation*}
with $\lambda'$ defined as above, then the series $G(z)$ satisfies
\begin{equation} \label{funeqG}
G(z) = \frac{1-z}{(1+z)^2} + qzG(qz).
\end{equation}
\end{lem}

\begin{proof}
We follow the same scheme as in the proof of Lemma~\ref{funeq1}. We have
\begin{equation*}\begin{split}
       \tfrac{1}
     {1 +z - \tfrac{(1-wq)^2z}
     {1 +z - (1-wq)(1-wq^2)zX}} =
\tfrac{ (1-wq)(1-wq^2)zX-(1+z) }{ (1-wq^2)(1-wq)(1+z)zX + (1-wq)^2z- (1+z)^2 },
\end{split}\end{equation*}
so we can introduce the matrix
\begin{equation*}
N(w,z) = \begin{pmatrix}
(1-wq)(1-wq^2)z & -(1+z) \\[2mm]
(1-wq)(1-wq^2)(1+z)z & (1-wq)^2z- (1+z)^2 \\
\end{pmatrix},
\end{equation*}
and it follows that the continued fraction $T_{\lambda'}(z)$ can
be written as an infinite product in the following way:
\begin{equation} \label{prod1}
T_{\lambda'}(z) =  N(1,z) N(q,z) N(q^2,z) \dotsm .
\end{equation}
Furthermore, let us define two other matrices:
\begin{equation*}
P(z)= \begin{pmatrix}
(q-1)z & 1       \\[2mm]
   0   & (1+z)^2 \\
\end{pmatrix},\qquad
R = \begin{pmatrix}
qz(1+z)^2 & 1-z  \\[2mm]
   0   & (1+z)^2 \\
\end{pmatrix}.
\end{equation*}
We have then
\begin{equation*}
  G(z) = P(z)  N(1,z) N(q,z) N(q^2,z) \dotsm.
\end{equation*}
Using matrices, the functional equation that we have to prove can be written:
\begin{equation} \label{funeqm}
 P(z) N(1,z) N(q,z) N(q^2,z) {\dotsm} = R P(qz) N(1,qz) N(q,qz) N(q^2,qz) \dotsm .
\end{equation}
As in the case of Lemma~\ref{funeq1} and $\Omega_n$, we can form a product of matrices that behaves nicely. Let
\begin{equation} \label{def_lambda}
\Lambda_n = N(q^{n-1},z)^{-1} \dotsm  N(1,z)^{-1} P(z)^{-1} R P(qz) N(1,qz) \dotsm N(q^{n-1},qz),
\end{equation}
its explicit form being given separately in Lemma~\ref{lem_lambda} below. Let $w_n=\Lambda_n [0] $,
from the explicit form in the lemma we have
\begin{equation*}
w_n = \frac{1-qz^2}{q^{n+2}z^2+q^{n+1}z^2+2q^{n+1}z+1-qz^2}.
\end{equation*}
The important point is that $w_n$ has no term of negative degree in $z$ in its expansion, and we can
apply Lemma~\ref{lemconv}. From the definition of $\Lambda_n$, we have:
\begin{equation} \label{matrixeq}
  P(z) N(1,z) \dotsm N(q^{n-1},z) [ w_n ]  = R P(qz) N(1,qz) \dotsm N(q^{n-1},qz) [ 0 ] ,
\end{equation}
and letting $n$ tend to infinity in the previous identity proves the functional equation \eqref{funeqm}
which was just a rewriting of \eqref{funeqG}.
\end{proof}

\begin{lem} \label{lem_lambda}
We have
\begin{equation}\label{lambda_is}
  \Lambda_n
  =\left(\begin{mmatrix}
    (q^{n+1}+2zq^{n+1}+q^{n}+z^2q-1)qz  &  1-qz^2   \\
    (1-q^{n})(1-q^{n+1})(qz^2-1)qz            &  q^{n+2}z^2+q^{n+1}z^2+2q^{n+1}z+1-qz^2 \\
  \end{mmatrix}\right).
\end{equation}
\end{lem}

\begin{proof}
As a consequence of the definition of $\Lambda_n$ in \eqref{def_lambda}, we have the recurrence relation:
\begin{equation*}
  \Lambda_{n+1} = N(q^n,z)^{-1} \Lambda_{n} N(q^n,qz).
\end{equation*}
It is possible to check that the expression in the right-hand side of \eqref{lambda_is} satisfies this relation and coincides
with $\Lambda_0$ when $n=0$. All comments concerning the calculation of $\Omega_n$ in Lemma~\ref{lem_omega}
and its proof apply as well in this case, so once again we omit the voluminous details and suggest the help of a
computer for a verification of the expression in \eqref{lambda_is}.
\end{proof}

\begin{remark}
The reader might have noticed that the right-hand side of \eqref{deffz} is a basic
hypergeometric series, more precisely a ${}_4\phi_4 $ in the notation of \cite{GR}.
There are established methods to prove continued fraction expansions of such series,
using {\it contiguous relations}. One of the most general result of this kind is
by Masson \cite{masson}. In the present case, we did not see how to apply or adapt these
known results, and instead had to use the fact that our series satisfy a simple
functional equation.
\end{remark}

\section{\texorpdfstring{Connection with Jacobi's identity for $(q;q)_\infty^3$}{Connection with Jacobi's identity for (q;q)³}}

Using the argument as in the first part of this article, we show here that from
Theorem~\ref{qgeno} we can obtain the following identity, due to Jacobi:
\begin{equation}\label{eq:cubed}
  \prod_{i\geq1} (1-q^i)^3 = \sum_{i=0}^{\infty} (-1)^i(2i+1) q^{\frac{i(i+1)}{2}}.
\end{equation}

Let $\Gg_1 = (-q,-q^2,-q^2,-q^3,-q^3,\dots)$ and $\Gg_2 =
(-q,-q,-q^2,-q^2,\dots)$.
We have shown in the previous section that $Y_k=W_k$ where the generating function of $W_k$
is a T-fraction. The lattice path interpretation of this T-fraction shows that
\[ 
Y_k(q) = \sum_{p\in\MD^*_{k}} \wt(p;\Gg_1,\Gg_2).
\]

We can reformulate $Y_k(q)$ using \dkc.s. For $(\lambda,A)\in \Delta_k^+$, we
define $\wt'_{q}(\lambda,A)$ to be
\[
\wt'_{q}(\lambda,A) = q^{|\lambda|} \prod_{u\in A} (-q^{g(u)}),
\]
where $g(u)=1+\flr{|u|/2}$ if $u$ is a horizontal arrow and $g(u)=\lceil |u|/2 \rceil$
if $u$ is a vertical arrow.  Recall the bijection between $\MD^*_{k}$ and
$\Delta_k^+$ explained in Section~\ref{sec:dk_config}. It is not difficult to
see that if $p\in\MD^*_{k}$ corresponds to $(\lambda,A)\in \Delta_k^+$, then
\[
\wt(p;\Gg_1,\Gg_2) = q^{\binom{k+2}{2} -1} \wt'_{q^{-1}}(\lambda,A),
\]

Thus
\begin{equation}
  \label{eq:11}
\sum_{C\in \Delta_k^+} \wt'_{q}(C)
= q^{\binom{k+2}{2} -1} Y_k(q^{-1}).
\end{equation}

Using the same argument sending $k\to\infty$ as we did in
Section~\ref{sec:limit}, we get
\begin{equation}
  \label{eq:9}
\lim_{k\to\infty} \sum_{C\in \Delta_k^+} \wt'_q(C) = \prod_{i\geq1}
\frac{1}{1-q^i} \prod_{i\geq1} (1-q^i)^2 \frac{1}{1-q} \prod_{i\geq1} (1-q^i)^2
=\frac{1}{1-q} \prod_{i\geq1} (1-q^i)^3.
\end{equation}

Using \eqref{eq:11} and \eqref{eq:9}, we have
\begin{align*}
\prod_{i\geq1} (1-q^i)^3  
&= \lim_{k\to\infty} (1-q) \sum_{C\in \Delta_k^+} \wt'_q(C)
= \lim_{k\to\infty} (1-q) q^{\binom{k+2}{2} -1} Y_k(q^{-1})\\
&= \lim_{k\to\infty} (q^{-1}-1) q^{\binom{k+2}{2}}
 \frac{ P_{k-1}(q^{-1}) + 2P_{k}(q^{-1}) + P_{k+1}(q^{-1}) } {q^{-1}-1}\\
&=\lim_{k\to\infty} \sum_{i=0}^k (-1)^i(2i+1) q^{\binom{i+1}{2}} 
\left(
q^{\binom{k+2}{2} - \binom{k}{2}} + 2 q^{\binom{k+2}{2} - \binom{k+1}{2}} 
+q^{\binom{k+2}{2} - \binom{k+2}{2}} 
\right)\\
&=\sum_{i=0}^\infty (-1)^i(2i+1) q^{\binom{i+1}{2}}.
\end{align*}
This finishes the proof of \eqref{eq:cubed}.


\part*{Final remarks and open problems}

\begin{remark}
As mentioned in the introduction, Touchard-Riordan--like formulas appear as moments of orthogonal polynomials.
It can be checked that the quantities considered in this article are related with the so-called {\it continuous dual
$q$-Hahn polynomials}, denoted $p_n(x;a,b,c|q)$. This is a family of polynomials in $x$ depending on four parameters
$a$, $b$, $c$, and $q$, see \cite{koekoek} for the definition and let $\mu_n(a,b,c|q)$ denote the $n$th
moment of this sequence. The moment generating function of orthogonal polynomials has a continued
fraction expansion (as a J-fraction) where the parameters are simply related with the coefficients in the three-term
recurrence relation, see \cite[Chapter 7]{aigner}. Using the first contraction in Lemma~\ref{contractions}, we can
always transform an S-fraction into a J-fraction, so that it is in theory straightforward to identify our S-fractions and
moment generating functions. Note also that the binomial transform $a^{-n} \sum_{k=0}^n  \binom nk  (-b)^{n-k} \mu_k $
gives the moments of the rescaled polynomials $p_n(ax+b)$. For the continued fraction in \eqref{fjtp}, the result is:
\begin{equation} \label{cdqh1}
\sum_{k=0}^n \left(\tbinom{2n}{n-k}-\tbinom{2n}{n-k-1}  \right) \sum_{j=-k}^{k} y^j q^{k(k+1)-j^2}
= 2^n \sum_{k=0}^n \binom nk (-1)^k  \mu_k(1,-yq,-y^{-1}q|q^2),
\end{equation} 
so that this quantity is the $n$th moment for the orthogonal polynomials
$p_n(1-\frac x2;1,-yq,-y^{-1}q|q^2)$. As for the $q$-Genocchi numbers, we have:
\begin{equation} \label{cdqh2}
  G_{2n+2}(q) = \frac{2^n}{(1-q)^{2n}} \sum_{k=0}^n \binom nk \mu_k(-q,-q,-q|q),
\end{equation}
so that $(1-q)^{2n}G_{2n+2}(q)$ is the $n$th moment for the orthogonal polynomials
$p_n(\frac x2-1;-q,-q,-q|q)$.
These continuous dual $q$-Hahn polynomials are the special case $d=0$ in the
sequence of Askey-Wilson polynomials \cite{koekoek}. A closed formula
for the moments $\mu_n(a,b,c|q)$ was given at the end of \cite{corteel2}, in the form
of a sum over four indices containing three $q$-binomial coefficients. It might be possible
to obtain our formulas by using the one from \cite{corteel2} and simplifying the right-hand
sides of \eqref{cdqh1} and \eqref{cdqh2}, but this would surely give rise to lengthy calculations.
\end{remark}

\begin{remark}
A general method for proving a continued fraction expansion is to use Hankel determinants
(this is related with orthogonal polynomials, see \cite{aigner,cuyt} for example).
Let $\{m_n\}_{n\geq 0}$ be a sequence with $m_0=1$, let $M_n$ be the matrix with
coefficients $(m_{i+j})_{0\leq i,j \leq n-1}$, and $M'_n$ with coefficients $(m_{i+j+1})_{0\leq i,j \leq n-1}$.
Then, provided that all the determinants are non-zero, we have $\sum_{n=0}^\infty m_n z^n= S_\lambda(z)$ 
where for any $n\geq 1$, 
\begin{equation*}
   \lambda_{2n-1} = \frac{ \det(M'_n) \det(M_{n-1})  }{  \det(M_n ) \det(M'_{n-1}) }, \qquad
   \lambda_{2n} = \frac{ \det(M'_{n-1}) \det(M_{n+1})  }{  \det(M_n ) \det(M'_{n}) }.
\end{equation*}
See \cite{jones}. When $m_n$ is equal to the left-hand side of \eqref{cdqh1}, it follows that
Theorem~\ref{thm:main} is equivalent to the evaluations:
\begin{align*}
  \det(M_n) &= \prod_{i=1}^{n-1}  \big(  (1+yq^{2i-1})(1+y^{-1}q^{2i-1})(1-q^{2i})^2 \big)^{n-i},  \\
  \det(M'_n) &= \prod_{i=1}^{n}  \big(  (1+yq^{2i-1})(1+y^{-1}q^{2i-1}) \big)^{n+1-i}   (1-q^{2i})^{2(n-i)},
\end{align*}
which might be proved by examining the vanishing locus. Similarly, there is a product-form for these determinants in
the case where $m_n=G_{2n+2}(q)$. However, it is unclear whether proving these evaluations could be simpler than
our method with the functional equation.
\end{remark}

\begin{remark}
We have seen throughout this work that the crucial property of T-fractions (as opposed to S-fractions) is that they
satisfy some functional equations linking $T_\lambda(z)$ and $T_\lambda(zq)$ or $T_\lambda(zq^2)$. We have no
enlightening explanation for this property, but we have a relevant observation. Consider the tails of the continued fraction
$T_\lambda(z)$, where the $m$th tail is the T-fraction associated with the shifted sequence $\{\lambda_n\}_{n\geq m}$.
In the various cases we have considered, $\lambda_n$ tends to $1$ as $n$ tends to $\infty$, and consequently the
tails converge to
\begin{equation*}
\frac{1}{1+z} \cmo \frac{z}{1+z} \cmo \frac{z}{1+z} \cmo \frac{z}{1+z} \cmo \dotsb .
\end{equation*}
But this is also equal to
\begin{equation*}
\frac{1}{1+qz} \cmo \frac{qz}{1+qz} \cmo \frac{qz}{1+qz} \cmo \frac{qz}{1+qz} \cmo \dotsb ,
\end{equation*}
for the simple reason that both continued fractions are equal to 1. In some sense, the substitution
$z\to zq$ has virtually no effect on the tails, and this is a clue to the fact that some T-fractions satisfy
functional equations.
\end{remark}

\begin{prob}
Some other results in the style of the triple product identify are \cite{berndt}:
\begin{align}   \label{berndt1}
  \sum_{j=-\infty}^\infty (3j+1) q^{3j^2+2j} &= \prod_{i\geq1} (1-q^i)(1-q^{2i-1})(1-q^{4i}), \\
  \sum_{j=-\infty}^\infty (6j+1) q^{3j^2+j}   &= \prod_{i\geq1} (1-q^{2i})^3(1-q^{4i-2})^2, \label{berndt2}
\end{align}
and it is natural to ask if some continued fractions are related. For example, if we define
\begin{equation*}
 \sum_{n=0}^\infty \xi_n(q) z^n = T_\lambda(z),
\quad \hbox{where} \quad
\begin{cases}
  \lambda_{3n+1} = (1-q^{2n+1})^3, \\
  \lambda_{3n+2} = (1-q^{2n+1})(1-q^{2n+2})^2, \\
  \lambda_{3n+3} = (1-q^{2n+2})^2(1-q^{2n+3}), \\
\end{cases}
\end{equation*}
then using a limit argument in paths (somewhat similar to what we did on $\delta_k$-configurations), it is possible to show that
\begin{equation*}
 \lim_{k\to\infty}  q^{2k(k+2)} (-1)^k (1-q^2) \xi_k(q^{-2}) = \prod_{i\geq1} (1-q^{2i})^3(1-q^{4i-2})^2,
\end{equation*}
which is the right-hand side of \eqref{berndt2}.
However we do not find any formula for $\xi_k(q)$. The problem is to find some results
in the style of our finite version of the triple product identify, related to \eqref{berndt1} or \eqref{berndt2}.
These two identities seem more likely to be treated rather than results such as
the ${}_1\Psi_1$ summation formula or the quintuple product \cite{berndt}.
\end{prob}

\begin{prob}
Is there a bijective proof showing that $\sum_{j=-k}^k y^j q^{j^2}$ counts $\delta_k$-configurations with the
$(y,q)$-weight? Going through the functional equation as in the proof of Theorem~\ref{jtp_th} is somewhat
unsatisfactory, mainly because there is no enlightening proof  of the fact that the T-fraction satisfies this
equation. The case $y=-1$ treated in the first part of this article was already quite involved, but maybe
some new ideas can give the general case.
\end{prob}

\begin{prob}
Can we find a good combinatorial model of $\mu_n(a,b,q)$ from which both the ordinary generating function and the
exponential one (as in Theorem~\ref{theoS}) can be derived combinatorially? It seems difficult in such generality.
For example,  little is known about the combinatorics of integers $\mu_n(a,b,1)$ when $a$ and $b$ are half-integers.
\end{prob}

\bigskip

\noindent
{\bf Acknowledgement}

\noindent
This work started when both authors were in Paris in the group of Sylvie Corteel.
We thank her for discussions and encouragements.

\renewcommand*{\refname}{}

\part*{References} 

\vspace{-7mm}

\bibliographystyle{plain}

\begin{thebibliography}{ccc}

\bibitem{aigner}
 M. Aigner, A course in enumeration, Springer, 2007.

\bibitem{alladi}
K. Alladi and A. Berkovich, New polynomial analogues of Jacobi's triple product and Lebesgue's identities,
Adv. in Appl. Math. 32 (2004), 801--824

\bibitem{berndt}
 B. Berndt, Number theory in the spirit of Ramanujan, American Mathmetical Society, 2006.

\bibitem{BJS}
 A. Burstein, M. Josuat-Verg\`es and W. Stromquist, New Dumont permutations, preprint (2010).

\bibitem{comtet}
 L. Comtet, Advanced Combinatorics, Reidel Publishing Company, 1974.

\bibitem{corteel}
S. Corteel and J. Lovejoy, Overpartitions, Trans. Amer. Math. Soc. 356 (2004), 1623--1635.

\bibitem{corteel2}
S. Corteel, R. Stanley, D. Stanton, L.K. Williams, Formulae for Askey-Wilson moments and enumeration
of staircase tableaux, preprint (2010).

\bibitem{cuyt}
A. Cuyt, A. B. Petersen, B. Verdonk, H. Waadeland and W.B. Jones,
Handbook of continued fractions for special functions, Springer, 2008.

\bibitem{dumont}
 D. Dumont, Further triangles of Seidel-Arnold type and continued fractions related to
 Euler and Springer numbers, Adv. in App. Math. 16 (1995), 275--296.

\bibitem{flajolet}
 P. Flajolet and H. Prodinger, personal communication.

\bibitem{GR}
 G. Gasper and M. Rahman, Basic hypergeometric series, Cambridge University Press, 1990.

\bibitem{goulden}
 I.P. Goulden and D.M. Jackson, Combinatorial Enumeration, Wiley, 1983.

\bibitem{ismail}
 M.E.H. Ismail, D. Stanton and X. G. Viennot, The combinatorics of $q$-Hermite polynomials and the
 Askey-Wilson integral, European J. Combin. 8 (1987), 379--392.

\bibitem{jones}
 W.B. Jones and W.J. Thron, Continued Fractions: Analytic Theory and Applications,
 Addison-Wesley, 1980.

\bibitem{josuat1}
M. Josuat-Verg\`es, A $q$-enumeration of alternating permutations,
European J. Combin. 31 (2010), 1892--1906.


\bibitem{josuat2}
M. Josuat-Verg\`es, \'Enum\'eration de tableaux et de chemins, moments de polyn\^omes orthogonaux,
PhD. thesis, Universit\'e d'Orsay, 2010.

\bibitem{rubey}
M. Josuat-Verg\`es and Martin Rubey, Crossings, Motzkin paths and moments, preprint (2010).

\bibitem{koekoek}
 R. Koekoek, P.A. Lesky and R.F. Swarttouw,
 Hypergeometric orthogonal polynomials and their $q$-analogues,
 Springer, 2010.

\bibitem{masson}
D.R. Masson, The last of the hypergeometric continued fractions, Contemp. Math. 190 (1995), 287--294.

\bibitem{penaud}
 J.-G. Penaud, Une preuve bijective d'une formule de Touchard-Riordan, Discrete Math. 139 (1995), 347--360.

\bibitem{read}
 R.C. Read, The chord intersection problem, Ann. N. Y. Acad. Sci. 139 (1979), 444--454.

\bibitem{roblet}
 E. Roblet and X.G. Viennot, Th\'eorie combinatoire des T-fractions et approximants
 de Pad\'e en deux points, Discrete Math. 153 (1996), 271--288.

\bibitem{riordan68}
 J. Riordan, Combinatorial identities, Wiley, 1968.

\bibitem{riordan}
 J. Riordan, The distribution of crossings of chords joining pairs of $2n$ points on a circle,
 Math. Comput. 29(129) (1975), 215--222.

\bibitem{schlosser}
 M. Schlosser,
 Abel-Rothe type generalizations of Jacobi's triple product identity, Dev. Math. 13 (2005), 383--400

\bibitem{stieltjes}
 T.J. Stieltjes, Sur quelques int\'egrales d\'efinies et leur d\'eveloppement en fractions continues,
 Quart. J. Math. 24 (1890), 370--382; \OE uvres compl\`etes 2, Noordhoff and Groningen, 1918, pp. 378--394.

\bibitem{touchard}
 J. Touchard, Sur un probl\`eme de configurations et sur les fractions continues,
 Canad. J. Math. 4 (1952), 2--25.

\bibitem{vahlen}
 K.T. Vahlen, Beiträge zu einer additiven Zahlentheorie,
 J. Reine Angew. Math. 112 (1893), 1--36.

\bibitem{viennot}
 X.G. Viennot, Interpr\'etations combinatoires des nombres d'Euler et de Genocchi,
 S\'eminaire de th\'eorie des nombres de l'Universit\'e Bordeaux I, 1981.

\bibitem{wall}
 H.S. Wall, Analytic theory of continued fractions, Van Nostrand, 1948.

\bibitem{warnaar}
 S.O. Warnaar, $q$-Hypergeometric proofs of polynomial analogues of the triple product identity, Lebesgue's
 identity and Euler's pentagonal number theorem, Ramanujan J. 8 (2005), 467--474.

\end{thebibliography}

\end{document}